\documentclass{amsart}   	
\input xypic 
\input xy 
\usepackage{amssymb,amsthm}
\usepackage{amsthm}
\usepackage{yhmath}
\usepackage{graphicx}
\usepackage{enumitem}
\usepackage{tikz}
\usepackage{tikz-cd}
\usetikzlibrary{calc}
\usetikzlibrary{decorations.pathmorphing}
\usepackage{hyperref}
\usepackage{adjustbox}
\usepackage{mathtools}
\usepackage[all]{xy}

\tikzset{curve/.style={settings={#1},to path={(\tikztostart)
    .. controls ($(\tikztostart)!\pv{pos}!(\tikztotarget)!\pv{height}!270:(\tikztotarget)$)
    and ($(\tikztostart)!1-\pv{pos}!(\tikztotarget)!\pv{height}!270:(\tikztotarget)$)
    .. (\tikztotarget)\tikztonodes}},
    settings/.code={\tikzset{quiver/.cd,#1}
        \def\pv##1{\pgfkeysvalueof{/tikz/quiver/##1}}},
    quiver/.cd,pos/.initial=0.35,height/.initial=0}

\tikzset{tail reversed/.code={\pgfsetarrowsstart{tikzcd to}}}
\tikzset{2tail/.code={\pgfsetarrowsstart{Implies[reversed]}}}
\tikzset{2tail reversed/.code={\pgfsetarrowsstart{Implies}}}
\tikzset{no body/.style={/tikz/dash pattern=on 0 off 1mm}}

\graphicspath{ {./images/} }

\newtheorem{theorem}{Theorem}[section]
\newtheorem{proposition}[theorem]{Proposition}
\newtheorem{lemma}[theorem]{Lemma}
\newtheorem{corollary}[theorem]{Corollary}

\theoremstyle{definition}
\newtheorem{definition}[theorem]{Definition}
\newtheorem{remark}[theorem]{Remark}
\newtheorem{example}[theorem]{Example}

\newcommand{\conn}{\ensuremath{\#}}

\newcommand{\zk}{\ensuremath{\mathcal{Z}_{K}}}

\newcommand{\hlgy}[1]{\ensuremath{H_{*}(#1)}}
\newcommand{\rhlgy}[1]{\ensuremath{\widetilde{H}_{*}(#1)}}

\newcounter{bean}

\newcommand{\namedright}[3]{\ensuremath{#1\stackrel{#2}
 {\longrightarrow}#3}}
\newcommand{\nameddright}[5]{\ensuremath{#1\stackrel{#2}
 {\longrightarrow}#3\stackrel{#4}{\longrightarrow}#5}}
\newcommand{\namedddright}[7]{\ensuremath{#1\stackrel{#2}
 {\longrightarrow}#3\stackrel{#4}{\longrightarrow}#5
  \stackrel{#6}{\longrightarrow}#7}}

\newcommand{\larrow}{\relbar\!\!\relbar\!\!\rightarrow}
\newcommand{\llarrow}{\relbar\!\!\relbar\!\!\larrow}
\newcommand{\lllarrow}{\relbar\!\!\relbar\!\!\llarrow}

\newcommand{\lnamedright}[3]{\ensuremath{#1\stackrel{#2}
 {\larrow}#3}}
\newcommand{\lnameddright}[5]{\ensuremath{#1\stackrel{#2}
 {\larrow}#3\stackrel{#4}{\larrow}#5}}

\newcommand{\llnameddright}[5]{\ensuremath{#1\stackrel{#2}
 {\llarrow}#3\stackrel{#4}{\llarrow}#5}}

\newcommand{\lllnameddright}[5]{\ensuremath{#1\stackrel{#2}
 {\lllarrow}#3\stackrel{#4}{\lllarrow}#5}}

\newcommand{\qqed}{\hfill\Box}

\usepackage[utf8]{inputenc}

\title[Loop spaces of Poincar\'{e} Duality complexes] 
{Loop spaces of $n$-dimensional Poincar\'e duality complexes whose $(n-1)$-skeleton 
    is a co-$H$-space}

\author{Lewis Stanton} 
\address{Mathematical Sciences, University of Southampton, Southampton SO17 1BJ, United Kingdom}
\email{lrs1g18@soton.ac.uk}

\author{Stephen Theriault}
\address{Mathematical Sciences, University of Southampton, Southampton SO17 1BJ, United Kingdom}
\email{s.d.theriault@soton.ac.uk}


\begin{document}

\begin{abstract}
Under certain hypotheses, we prove a loop space decomposition for simply-connected Poincar\'e Duality complexes of dimension~$n$ whose $(n-1)$-skeleton is a co-$H$-space. This unifies many known decompositions 
obtained in different contexts and establishes many new families of examples. As consequences, 
we show that such a looped Poincar\'{e} Duality complex retracts off the loops of its 
$(n-1)$-skeleton and describe its homology as a one-relator algebra.
\end{abstract}

\maketitle

\section{Introduction}

A simply-connected $CW$-complex $M$ is a \emph{Poincar\'{e} Duality complex} if the cohomology of $M$ satisfies Poincar\'{e} Duality. Examples include simply-connected closed $n$-dimensional manifolds. There has been a great deal of activity recently in studying the homotopy theory of Poincar\'{e} Duality complexes. This often takes the form of a loop space decomposition: if $\Omega M$ is the based loop space of $M$ then the goal is to show that $\Omega M$ is homotopy equivalent to a product of other spaces. One consequence is that the homotopy groups of $M$ can then be described in terms of the homotopy groups of the factors. Ideally, the factors are recognisable spaces whose homotopy groups are known through a range or have appealing global properties. 

For example, using different methods, in~\cite{BB1,BT1} it was shown that if $M$ is a simply-connected $4$-manifold and the rank of $H^{2}(M)$ is at least $2$ then $\Omega M$ is homotopy equivalent to an explicit product of spheres and loops on spheres. Consequently, the homotopy groups of a simply-connected $4$-manifold can be determined to the same extent as the homotopy groups of spheres. Other families of manifolds for which loop space decompositions into recognisable factors are known include \mbox{$(n-1)$}-connected $2n$-dimensional manifolds~\cite{BB1,BT1} and $(n-1)$-connected $(2n+1)$-dimensional manifolds where either $H^{n}(M)$ is torsion-free and of rank at least one~\cite{Bas,BT2} or $H^{n}(M)$ has torsion but $n$ is even~\cite{BW,HT1}.

In general, let $\overline{M}$ be the $(n-1)$-skeleton of $M$. In the cases 
mentioned and in many other known homotopy decompositions for $\Omega M$ it turns out 
that $\overline{M}$ is a co-$H$-space (in fact, usually a suspension). This led J. Wu to 
ask if there is a general decomposition formula for $\Omega M$ if~$\overline{M}$ is a 
co-$H$-space. The purpose of this paper is to show that this is true, given an extra hypothesis. 

The hypothesis is that if~$M$ is $(m-1)$-connected then there is a map 
\(\namedright{S^{m}}{}{M}\) 
with a left homotopy inverse. For $m<n$, as $\overline{M}$ is a co-$H$-space, this implies 
that there is a homotopy equivalence $\overline{M}\simeq S^{m}\vee A$ for some space $A$, 
so the hypothesis can be interpreted as a sort of partial freeness condition. To describe the 
decomposition another space is needed. Poincar\'{e} Duality implies that $\overline{M}$ 
has dimension $n-m$, so the same is true for $A$. It also implies that in 
$\overline{M}\simeq S^{m}\vee A$ there is a class $y\in H^{n-m}(A)$ whose cup product 
with the cohomology class corresponding to $S^{m}$ equals the generator of $H^{n}(M)$. 
It will be shown that there is a space $B$ and a homotopy cofibration 
\(\nameddright{B}{}{A}{}{S^{n-m}}\), 
where $S^{n-m}$ corresponds to $y$.

\begin{theorem} 
   \label{main} 
   Let $M$ be an $(m-1)$-connected $n$-dimensional Poincar\'{e} Duality complex where $2\leq m<n$. Suppose that $\overline{M}$ is a co-$H$-space and there is a map 
   \(\namedright{S^{m}}{}{M}\) 
   with a left homotopy inverse 
   \(\namedright{M}{}{S^{m}}\). 
   Then there is a homotopy fibration 
   \[\nameddright{A\vee(B\wedge\Omega S^{m})}{}{M}{}{S^{m}}\] 
   that splits after looping to give a homotopy equivalence 
   \[\Omega M\simeq\Omega S^{m}\times\Omega(A\vee (B\wedge\Omega S^{m})).\]
\end{theorem} 

This decomposition also satisfies a naturality property which is discussed in detail in Remark~\ref{natremark}.
We also show that if $\overline{M}$ is homotopy equivalent to a wedge 
of spheres and Moore spaces then so are $A$ and $B$, and their homotopy types can be 
simply read off from the homology of $M$. Thus in these cases the homotopy type of $\Omega M$ is very 
explicit and completely determined by the homology of $M$. 

Theorem~\ref{main} unifies a wide range of decomposition results. 
Manifolds satisfying the hypotheses of Theorem ~\ref{main} include $(n-1)$-connected, $(2n)$-dimensional manifolds with $n \notin \{2,4,8\}$ and the $(n-1)$-connected $(2n+1)$-dimensional manifolds described above, as do connected sums of products of two spheres 
or connected sums of nontrivial $S^{n-m}$-bundles over $S^{m}$ with \mbox{$m< n-1$.} 
A new case is a moment-angle manifold associated to a neighbourly simplicial complex. 

We go on to significantly extend the known families of examples by showing that the hypotheses of 
Theorem~\ref{main} are preserved under the operations of connected sum and gyration (a type of surgery). Explicitly, let $\mathcal{A}$ be the family of Poincar\'{e} Duality complexes that satisfy the hypotheses of Theorem~\ref{main}. We show that if $M\in\mathcal{A}$ and $N$ is a simply-connected Poincar\'{e} Duality complex such that~$\overline{N}$ is a co-$H$-space and the connectivity of $N$ is at least that of $M$, then the connected sum $M\conn N$ is in $\mathcal{A}$. We also show that if $M\in\mathcal{A}$ and $\mathcal{G}^{k}_{\tau}(M)$ is a twisted $k$-gyration of $M$, then $\mathcal{G}^{k}_{\tau}(M)\in\mathcal{A}$. In particular, applying Theorem~\ref{main} gives an integral homotopy decomposition for $\Omega\mathcal{G}^{k}_{\tau}(M)$, improving on the local decomposition in~\cite{HT2} that inverted primes related to the $J$-homomorphism. 

Theorem~\ref{main} leads to other good properties 
of Poincar\'{e} Duality complexes in $\mathcal{A}$. One benefit is to describe the effect in homotopy of the attaching map for the $n$-cell of $M$. Recall that there is a homotopy cofibration 
\[\nameddright{S^{n-1}}{f}{\overline{M}}{i}{M}\] 
where $f$ is the attaching map for the $n$-cell of $M$ and $i$ is the skeletal inclusion. An approach is developed in~\cite{BT2} to study the homotopy theory of $M$ by studying properties of the spaces and maps in this homotopy cofibration. A particularly useful special case is when $\Omega i$ has a right homotopy inverse, that is, when $\Omega M$ retracts off $\Omega\overline{M}$. In this case the attaching map $f$ is called \emph{inert}, generalising a similar notion in rational homotopy theory~\cite{HL}. The inert property was shown in~\cite[Proposition 3.5]{BT2} to imply that there is a homotopy fibration 
\begin{equation} 
  \label{Mbarfib} 
  \nameddright{S^{n-1}\rtimes\Omega M}{}{\overline{M}}{}{M}  
\end{equation} 
that splits after looping to give a homotopy equivalence 
\begin{equation} 
  \label{Mbardecomp} 
  \Omega\overline{M}\simeq\Omega M\times\Omega(S^{n-1}\rtimes\Omega M). 
\end{equation} 
Here, for spaces $A$ and $B$, the \emph{right half-smash} $A\rtimes B=(A\times B)/\sim$ is the quotient space given by collapsing the subspace $\{\ast\}\times B$ to the basepoint, and if $A$ is a co-$H$-space then there is a homotopy equivalence $A\rtimes B\simeq A\vee (A\wedge B)$. Rationally, Halperin and Lemaire~\cite{HL} showed that the attaching map for the $n$-cell of any simply-connected $n$-dimensional Poincar\'{e} Duality complex is inert, provided the rational cohomology is generated by more than one element. The second author~\cite{T1} showed 
that if $M$ is $(m-1)$-connected and there is a map 
\(\namedright{S^{m}}{}{M}\) 
having a left homotopy inverse, then the attaching map for the $n$-cell in $M$ is 
\emph{integrally} inert. Thus every member of the class $\mathcal{A}$ has the following property.

\begin{theorem} 
   \label{introinert} 
   If $M$ is an $n$-dimensional Poincar\'{e} Duality complex in $\mathcal{A}$ then 
   the attaching map for the $n$-cell of $M$ is integrally inert.~$\qqed$ 
\end{theorem} 

Theorem~\ref{introinert} has useful consequences. One is that the homotopy decomposition of $\Omega M$ 
in Theorem~\ref{main} then refines the decomposition of $\Omega\overline{M}$ 
via~(\ref{Mbarfib}) and~(\ref{Mbardecomp}). Another is that Theorem~\ref{main} 
can now be used to identify new families of Poincar\'{e} Duality complexes whose attaching map for 
the top cell is integrally inert. This includes all gyrations $\mathcal{G}^{k}_{\tau}(M)$ with $M\in\mathcal{A}$, substantially adding to the examples in~\cite{Hu}, and simply-connected $6$-manifolds $M$ with regular circle action for which $H_{2}(M)$ has at least one $\mathbb{Z}$ summand. 

A second benefit of Theorem~\ref{main} is that the co-$H$-property of $\overline{M}$ 
combined with the inert property of the attaching map for the $n$-cell of $M$ 
leads to a calculation of $H_{\ast}(\Omega M;R)$ as an algebra. To state this, 
let $R$ be a commutative ring with unit and let $V$ be a free graded $R$-module. Write 
$\Sigma^{-1} V$ for the free graded $R$-module whose generators are shifted down one degree 
as compared to those in $V$. Let $T(\ \ )$ be the free tensor algebra functor. If $M$ is 
$n$-dimensional, let 
\(\namedright{S^{n-1}}{f}{\overline{M}}\) 
be the attaching map for the $n$-cell of $M$ and let 
\(\namedright{S^{n-2}}{\widetilde{f}}{\Omega\overline{M}}\) 
be its adjoint. 

\begin{theorem} 
   \label{introloophlgy} 
   Let $M$ be an $n$-dimensional Poincar\'{e} Duality complex in $\mathcal{A}$ and 
   let $R$ be a commutative ring with unit such that $\rhlgy{\overline{M};R}$ is a 
   free $R$-module. Then there is an isomorphism of algebras 
   \[\hlgy{\Omega M;R}\cong T(\Sigma^{-1}\rhlgy{\overline{M};R})/(\mathrm{Im}\,(\widetilde{f}_{\ast}))\] 
   where $(\mathrm{Im}\,(\widetilde{f}_{\ast}))$ is the two-sided ideal generated by 
   $\mathrm{Im}\,(\widetilde{f}_{\ast})$. 
\end{theorem}   

In particular, Theorem~\ref{introloophlgy} implies that $\hlgy{\Omega M;R}$ is a 
\emph{one-relator algebra}, a free algebra with only a single relation. Moreover, if $M$ satisfies the hypotheses of Theorem ~\ref{introloophlgy} with $n \leq 3m-2$ and $R = \mathbb{Q}$, the relation is quadratic (see Remark ~\ref{rmk:quadrel}).

The families of examples that satisfy Theorem~\ref{main} can be extended further by localising. They include highly connected Poincar\'e duality complexes and moment-angle manifolds associated to minimally non-Golod complexes. In particular, in the first case, we show the following. 

\begin{theorem}
\label{thm:introloopPDcomplextors}
   Let $M$ be an $(m-1)$-connected, closed Poincar\'e duality complex of dimension $n$, where $2 \leq m < n$ and $n \leq 3m-1$. Let $k$ be the least integer such that $H_k(M)$ contains a $\mathbb{Z}$ summand, and suppose that $k < n$. If $k$ is even and $k=m = n-m$, suppose there exists a generator $x \in H^k(M)$ such that $x^2 = 0$. Localise away from primes $p$ appearing as $p$-torsion in $H_*(M)$, and primes $p \leq \frac{n-k+3}{2}$ if $k$ is odd, or primes $p \leq \frac{n-k+4}{2}$ if $k$ is even. Then there is a homotopy fibration 
   \[\nameddright{A\vee(B\wedge\Omega S^{k})}{}{M}{h'}{S^{k}},\] where $A$ and $B$ are wedges of spheres that can be explicitly enumerated as in Corollary ~\ref{cor:MbarwedgeofsphereandMoore}.
   Moreover, this homotopy fibration splits after looping to give a homotopy equivalence 
   \[\Omega M\simeq\Omega S^{m}\times\Omega(A\vee(B\wedge\Omega S^{m})).\]
\end{theorem}

Theorem~\ref{thm:introloopPDcomplextors} modestly improves on 
 and brings under the unifying umbrella of Theorem~\ref{main} a result of Basu and Basu~\cite{BB2}, who gave a local decomposition of such an $M$ provided $n\leq 3m-2$. Their list of inverted primes and their method of proof is different. They localised away from primes appearing as torsion in the integral homology of $M$ and a set of primes depending on the image of the rational Hurewicz homomorphism. Their proof first calculated the local homology of $\Omega M$ and then used this as a guide to identify what the factors of $\Omega M$ should be. The advantages of our approach are that the primes that must be inverted are more easily described and the local homology of $\Omega M$ can be recovered topologically via Theorem~\ref{introloophlgy}. Moreover, in the dimensional range $n \leq 3m-2$, we give a local decomposition of $\Omega M$ that allows for large primes in homology (see Theorem ~\ref{thm:loopPDcomplextors}).  
\medskip 

This paper is organised as follows. Theorem~\ref{main} is proved in 
Section~\ref{sec:DecompofPDsusskel} and its refinement when~$\overline{M}$ is a 
wedge of spheres and Moore spaces is proved in Section~\ref{sec:Refinement}. 
Theorems~\ref{introinert} and~\ref{introloophlgy} are proved in Section~\ref{sec:inert}. 
Examples are given in Section~\ref{sec:examples} and these are significantly expanded 
on in Section~\ref{sec:connsumgyration} by showing the class $\mathcal{A}$ is closed under 
the connected sum and gyration operations. Sections~\ref{sec:localdecomphighlyconnCW} 
and~\ref{sec:highyconnPDcomplex} extend the integral results to local settings.

The authors thank the referee for several helpful comments which improved the paper and suggesting that a naturality statement be included.

\section{Loop space decompositions of Poincar\'e duality complexes with \texorpdfstring{$(n-1)$}{(n-1)}-skeleton a co-\texorpdfstring{$H$}{H}-space}
\label{sec:DecompofPDsusskel}

Let $M$ be a simply-connected, closed $n$-dimensional Poincar\'{e} Duality complex. As $M$ is simply-connected, 
it has a $CW$-structure with a single $n$-cell; fix such a $CW$-structure. Let $\overline{M}$ be the $(n-1)$-skeleton of $M$. Note that 
as $M$ is closed then $\overline{M}$ is homotopy equivalent to $M$ with a puncture. Observe that 
there is a homotopy cofibration 
\[\nameddright{S^{n-1}}{f}{\overline{M}}{i}{M}\] 
where $f$ attaches the $n$-cell to $M$ and $i$ is the inclusion of the $(n-1)$-skeleton. 

Suppose that $M$ is $(m-1)$-connected for some $2\leq m<n$. Note that if $H_m(M) \neq 0$ then $m<n$ implies that 
$M\not\simeq S^{n}$. As we proceed two hypotheses will be introduced: 
\begin{itemize} 
   \item $\overline{M}$ is a co-$H$-space;
   \item there is a map 
         \(\namedright{S^{m}}{s'}{M}\) 
         that has a left homotopy inverse 
         \(\namedright{M}{h'}{S^{m}}\). 
\end{itemize} In this section, a decomposition of $\Omega M$ is given under these hypotheses. Examples of families of such Poincar\'{e} Duality complexes will be given in Section ~\ref{sec:examples} and Section ~\ref{sec:connsumgyration}. 



Since $m<n$, the map 
\(\namedright{S^{m}}{s'}{M}\) 
factors through the $(n-1)$-skeleton of $M$ as a composite 
\[\nameddright{S^{m}}{s}{\overline{M}}{i}{M}\] 
for some map $s$. Notice that as $h'$ is a left homotopy inverse for $s'$, the composite 
\[h\colon\nameddright{\overline{M}}{i}{M}{h'}{S^{m}}\] 
is a left homotopy inverse for $s$. Define the space $A$ and the map $a$ by the homotopy cofibration 
\begin{equation} 
  \label{Mbarcofib} 
  \nameddright{S^{m}}{s}{\overline{M}}{a}{A}. 
\end{equation}  

\begin{lemma} 
   \label{Mbarsplit} 
   Suppose that $\overline{M}$ is a co-$H$-space with comultiplication $\sigma$. Then 
   the homotopy cofibration~(\ref{Mbarcofib}) splits to give a homotopy equivalence 
   \[e\colon\nameddright{\overline{M}}{\sigma}{\overline{M}\vee\overline{M}}{h\vee a}{S^{m}\vee A}.\] 
\end{lemma} 

\begin{proof} 
Since $h$ is a left homotopy inverse for $s$, the long exact sequence in homology induced by the homotopy 
cofibration~(\ref{Mbarcofib}) degenerates into split short exact sequences in each degree. This is 
geometrically realized by the composite $e$. Thus $e$ induces 
an isomorphism in homology. As all the spaces are simply-connected, $e$ is a homotopy 
equivalence by Whitehead's Theorem. 
\end{proof} 

In general, if $X$ and $Y$ are spaces, the \emph{right half-smash} is the quotient space 
\[X\rtimes Y=(X\times Y)/\sim\] 
where the subspace $\ast\times Y$ is collapsed to a point. Let 
\(p_{1}\colon\namedright{X\vee Y}{}{X}\) 
be the pinch map to the left wedge summand. As in~\cite[Theorem 7.7.7]{Sel}, a method  
developed by Ganea~\cite{G} proves the following. 

\begin{lemma} 
   \label{Ganea} 
   There is a natural homotopy fibration 
   \[\nameddright{Y\rtimes\Omega X}{}{X\vee Y}{p_{1}}{Y}.\eqno\qqed\] 
\end{lemma} 

In our case, consider $p_{1}\circ e$. The naturality of $p_{1}$ 
implies that $p_{1}\circ e=p_{1}\circ (h\vee a)\circ\sigma\simeq h\circ p_{1}\circ\sigma$. 
Since $\overline{M}$ is a co-$H$-space, $p_{1}\circ\sigma$ is homotopic to the identity map on $\overline{M}$. 
Thus $p_{1}\circ e\simeq h$. Therefore, if $E$ is the homotopy fibre of $h$, we obtain a homotopy fibration diagram 
\begin{equation} 
  \label{e'pb} 
  \diagram 
      E\rto^-{e'}\dto  & A\rtimes\Omega S^{m}\dto \\ 
      \overline{M}\rto^-{e}\dto^{h} & S^{m}\vee A\dto^{p_{1}} \\ 
      S^{m}\rdouble & S^{m} 
  \enddiagram 
\end{equation} 
for some map $e'$. Notice that the upper square is a homotopy pullback. Therefore, 
as $e$ is a homotopy equivalence, we immediately obtain the following. 

\begin{lemma} 
   \label{Etype} 
   The map 
   \(\namedright{E}{e'}{A\rtimes\Omega S^{m}}\) 
   is a homotopy equivalence.~$\qqed$  
\end{lemma} 

Consider the diagram   
\begin{equation} 
  \label{Mdata} 
  \diagram 
        & E\rto\dto & E'\dto \\ 
        S^{n-1}\rto^-{f} & \overline{M}\rto^-{i}\dto^{h} & M\dto^{h'} \\ 
        & S^{m}\rdouble &S^{m}.  
  \enddiagram 
\end{equation}     
Here, the middle row is a homotopy cofibration, the lower square homotopy commutes by definition of $h$ as the restriction of $h'$ to the $(n-1)$-skeleton, $E'$ is defined as the homotopy fibre of $h'$, and the upper square is an induced map of homotopy fibres. With such data, by ~\cite{BT1} there is a homotopy cofibration 
\begin{equation} 
   \label{BTcofib} 
   \nameddright{S^{n-1}\rtimes\Omega S^{m}}{\theta}{E}{}{E'} 
\end{equation} 
for some map $\theta$.
Under certain hypotheses, a left homotopy inverse for $\theta$ and a splitting of the homotopy 
cofibration~(\ref{BTcofib}) were constructed in~\cite{T2}. 

\begin{theorem} 
   \label{Tinert} 
   Suppose that $M$ is $(m-1)$-connected and there is a map 
   \(\namedright{S^{m}}{}{M}\) 
   with a left homotopy inverse 
   \(\namedright{M}{}{S^{m}}\). 
   Then the homotopy cofibration~(\ref{BTcofib}) splits to give a homotopy equivalence 
   $E\simeq E'\vee(S^{n-1}\rtimes\Omega S^{m})$.~$\qqed$ 
\end{theorem} 

One step in the proof of Theorem~\ref{Tinert} was to find a left homotopy inverse of $\theta$. 
In this paper we wish to be more careful about the choice of such a left homotopy inverse. So 
to go further, the argument proving Theorem~\ref{Tinert} is briefly summarized. Before proceeding, we first prove a lemma regarding the ring structure of $H^*(M)$.

\begin{lemma}
\label{lem:distinctgens}
    Let $M$ be an $(m-1)$-connected, $n$-dimensional Poincar\'e Duality complex where \mbox{$2 \leq m <n$}. Suppose there is a map $s\colon S^m \rightarrow M$ with a left homotopy inverse $h\colon M \rightarrow S^m$. If $\iota\in H^{m}(S^{m})$ is a generator, let $x=(h')^{\ast}(\iota)\in H^{m}(M)$. Then $x$ generates a primitive $\mathbb{Z}$-summand, $x^{2}=0$, and there is a $\mathbb{Z}$-generator $y \in H^{n-m}(M)$ with $x \neq y$ such that $x \cup y$ generates $H^n(M)$. Further, if $m=n-m$ then $y$ can be chosen to be in $H^{m}(M)\backslash\mathbb{Z}\{x\}$.
\end{lemma}
\begin{proof}
The primitivity of $x$ follows since $h\circ s$ is homotopic to the identity map on $S^{m}$. Since $(h')^{\ast}$ is an algebra map 
and $\iota^{2}=0$, it follows that $x^{2}=0$. Poincar\'e duality therefore implies that there exists $y \in H^{n-m}(M)$ with $x \neq y$ such that $x \cup y$ generates $H^n(M)$. If $m=n-m$ then possibly $y=x+z$ for some $z\in H^{m}(M)\backslash\mathbb{Z}\{x\}$. But then as $x^{2}=0$ we obtain $x\cup y=x\cup z$. Note that $z$ must also generate a $\mathbb{Z}$-summand; otherwise $z$ is rationally trivial, implying that $x\cup z$ is rationally trivial, a contradiction. Thus we may take $y$ to be $z$.
\end{proof}

Since $M$ is $(m-1)$-connected, by Poincar\'{e} Duality, $H^{k}(M)\cong 0$ for $n-m<k<n$. The 
universal coefficient theorem then implies that $H_{k}(M)\cong 0$ for $n-m<k<n$. Therefore 
the $(n-1)$-skeleton of $\overline{M}$ has dimension at most $n-m$. By assumption, there 
is a map 
\(\namedright{S^{m}}{s'}{M}\) 
with a left homotopy inverse 
\(\namedright{M}{h'}{S^{m}}\). 
By Lemma~\ref{lem:distinctgens}, $x=(h')^{\ast}(\iota)\in H^{m}(M)$ generates a primitive $\mathbb{Z}$-summand and there is a $\mathbb{Z}$-generator $y\in H^{n-m}(M)$ such that $x\neq y$ and $x\cup y$ generates $H^{n}(M)$. The universal coefficient theorem implies that $y$ dualizes to a $\mathbb{Z}$-summand in $H_{n-m}(M)$. Thus $\overline{M}$ is precisely $(n-m)$-dimensional. 
Let $\overline{\overline{M}}$ be the $(n-m-1)$-skeleton of $\overline{M}$. Then there is a homotopy cofibration 
\[\nameddright{\bigvee_{i=1}^{d} S^{n-m-1}}{}{\overline{\overline{M}}}{}{\overline{M}}\]  
that attaches the $(n-m)$-cells to $\overline{M}$. Note that $d\geq 1$. 

In general, any homotopy cofibration 
\(\nameddright{X}{}{Y}{}{Z}\) 
has a connecting map 
\(\delta\colon\namedright{Z}{}{\Sigma X}\) 
and a homotopy coaction 
\(\psi\colon\namedright{Z}{}{Z\vee\Sigma X}\) 
with the property that $\psi$ composed with the pinch map to $Z$ is homotopic to the identity map and 
$\psi$ composed with the pinch map to $\Sigma X$ is homotopic to $\delta$. In our case, we obtain 
a connecting map 
\[\namedright{\overline{M}}{\delta}{\bigvee_{i=1}^{d} S^{n-m}}\] 
and a homotopy coaction 
\[\namedright{\overline{M}}{\psi}{\overline{M}\vee\bigvee_{i=1}^{d} S^{n-m}}.\]  
The generator $y\in H^{n-m}(M)$ may also be regarded as a generator in $H^{n-m}(\overline{M})$, 
and for dimensional reasons it is in the image of $\delta^{\ast}$. Let 
\[p\colon\namedright{\bigvee_{i=1}^{d} S^{n-m}}{}{S^{n-m}}\]  
be the pinch map to the $i=1$ summand. Changing $\bigvee_{i=1}^{d} S^{n-m}$ 
by a self-equivalence if necessary, we may assume that the composite 
\[p'\colon\nameddright{\overline{M}}{\delta}{\bigvee_{i=1}^{d} S^{n-m}}{p}{S^{n-m}}\] 
has image $y$ in cohomology. Let $\psi'$ be the composite 
\[\psi'\colon\llnameddright{\overline{M}}{\psi}{\overline{M}\vee\bigvee_{i=1}^{d} S^{n-m}}{h\vee p} 
      {S^{m}\vee S^{n-m}}.\] 
As $\psi$ is a comultiplication, $\psi'$ composed with the pinch map to $S^{m}$ is homotopic to $h$ and 
$\psi'$ composed with the pinch map to $S^{n-m}$ is homotopic to $p'$. Thus $(\psi')^{\ast}$ sends the 
generator of $H^{m}(S^{m})$ to $x$ and the generator of $H^{n-m}(S^{n-m})$ to $y$. 

Generically, let 
\(\namedright{X\vee Y}{p_{1}}{X}\) 
be the pinch map to the left wedge summand. The naturality of~$p_{1}$  implies that 
$p_{1}\circ\psi'=p_{1}\circ(h\vee p)\circ\psi\simeq h\circ p_{1}\circ\psi$. Since $\psi$ is 
a homotopy coaction, $p_{1}\circ\psi$ is homotopic to the identity map on $\overline{M}$. 
Thus $p_{1}\circ\psi'\simeq h$. This homotopy results in a homotopy fibration diagram 
\begin{equation} 
  \label{psipb} 
  \diagram 
        E\rto^-{\gamma}\dto & S^{n-m}\rtimes\Omega S^{m}\dto \\ 
        \overline{M}\rto^-{\psi'}\dto^{h} & S^{m}\vee S^{n-m}\dto^{p_{1}} \\ 
        S^{m}\rdouble & S^{m} 
  \enddiagram 
\end{equation} 
that defines the map $\gamma$. Let 
\[q\colon\namedright{S^{n-m}\rtimes\Omega S^{m}}{}{S^{n-m}\wedge\Omega S^{m}}\] 
be the standard quotient map from the half-smash to the smash product. By the 
James construction~\cite{J}, there is a homotopy equivalence 
$\Sigma\Omega S^{m}\simeq\bigvee_{k=1}^{\infty}\Sigma S^{k(m-1)}$. Thus freely moving 
suspension coordinates gives homotopy equivalences 
\[\begin{split} 
    S^{n-m}\wedge\Omega S^{m} &  
       \simeq S^{n-m}\wedge(\bigvee_{k=1}^{\infty} S^{k(m-1)}) \\ 
    & \simeq (S^{n-m}\wedge S^{m-1})\vee(S^{n-m}\wedge(\bigvee_{k=2}^{\infty} S^{k(m-1)})) \\ 
    & \simeq (S^{n-m}\wedge S^{m-1})\vee (S^{n-m}\wedge S^{m-1}\wedge(\bigvee_{k=1}^{\infty} S^{k(m-1)})) \\
    & \simeq S^{n-1}\vee(S^{n-1}\wedge\Omega S^{m}) \\ 
    & \simeq S^{n-1}\rtimes\Omega S^{m}. 
 \end{split}\] 
In~\cite{T2} it was shown that the composite 
\begin{equation} 
  \label{thetainv} 
  \namedddright{S^{n-1}\rtimes\Omega S^{m}}{\theta}{E}{\gamma}{S^{n-m}\rtimes\Omega S^{m}}{q}{S^{n-m}\wedge\Omega S^{m}\simeq S^{n-1}\rtimes\Omega S^{m}} 
\end{equation} 
is a homotopy equivalence. 

By Lemma~\ref{Etype}, there is a homotopy equivalence $E\simeq A\rtimes\Omega S^{m}$. The 
goal is to show that $\gamma$ is well-behaved with respect to this homotopy equivalence in 
order to identify the homotopy type of the cofibre $E'$ of $\theta$. The first step is to determine to what extent the homotopy class of $\gamma$ is determined by the homotopy pullback~(\ref{psipb}). 

\begin{lemma} 
   \label{gammaclass} 
   The homotopy class of the map $\gamma$ in~(\ref{psipb}) is uniquely determined by it making the top square in~(\ref{psipb}) homotopy commute. 
\end{lemma} 

\begin{proof} 
We first show that $E$ is a co-$H$-space. By hypothesis, 
$\overline{M}$ is a co-$H$-space. By Lemma~\ref{Mbarsplit}, 
$\overline{M}\simeq S^{m}\vee A$, so as $A$ retracts off a 
co-$H$-space it is itself a co-$H$-space. In general, if $B$ is a co-$H$-space and $C$ is any space then $B\rtimes C$ is a co-$H$-space. Therefore the homotopy equvalence 
$E\simeq A\rtimes\Omega S^{m}$ in Lemma~\ref{Etype} implies 
that $E$ is a co-$H$-space. 

Now consider the homotopy fibration sequence 
\[\namedddright{\Omega S^{m}}{\delta}{S^{n-m}\rtimes\Omega S^{m}}{r}{S^{m}\vee S^{n-m}}{p_{1}}{S^{m}}\] 
where $r$ is a name for the map from the homotopy fibre to the total space and $\delta$ is the fibration connecting map. Since $p_{1}$ has a right homotopy inverse, $\delta$ is null homotopic.
Suppose that there is another map 
\(\namedright{E}{\gamma'}{S^{n-m}\rtimes\Omega S^{m}}\) 
that makes the top square in~(\ref{psipb}) homotopy commute. As $E$ is a co-$H$-space, we can consider the difference 
\(\lnamedright{E}{\gamma-\gamma'}{S^{n-m}\rtimes\Omega S^{m}}\). 
As both $\gamma$ and $\gamma'$ make the top square in~(\ref{psipb}) homotopy commute, the composite 
\[\lnameddright{E}{\gamma-\gamma'}{S^{n-m}\rtimes\Omega S^{m}}{r}{S^{m}\vee S^{n-m}}\] 
is null homotopic. Therefore $\gamma-\gamma'$ lifts to the homotopy fibre of $r$, meaning it lifts through $\delta$. But~$\delta$ is null homotopic, implying that $\gamma\simeq\gamma'$. 
\end{proof}

The next step is to reconcile the map 
\(\namedright{\overline{M}}{\psi'}{S^{m}\vee S^{n-m}}\) 
used to define~$\gamma$ in~(\ref{psipb}) with the comultiplication on $\overline{M}$ used 
to produce the homotopy equivalence in Lemma~\ref{Etype}. In general, if 
\mbox{\(\namedddright{X}{}{Y}{}{Z}{\delta}{\Sigma X}\)} 
is a homotopy cofibration sequence and $Z$ is a co-$H$-space with comultiplication~$\sigma$, 
then the associated homotopy coaction 
\(\namedright{Z}{\psi}{Z\vee\Sigma A}\) 
need not be homotopic to the composite 
\(\nameddright{Z}{\sigma}{Z\vee Z}{1\vee\delta}{Z\vee\Sigma X}\). 
The following lemma overcomes this in our case. 

\begin{lemma} 
   \label{comatch} 
   Let $\overline{M}$ be a co-$H$-space with comultiplication $\sigma$. Then the map 
   \(\namedright{\overline{M}}{\psi'}{S^{m}\vee  S^{n-m}}\) 
   is homotopic to the composite  
   \(\nameddright{\overline{M}}{\sigma}{\overline{M}\vee\overline{M}}{h\vee p'}{S^{m}\vee S^{n-m}}\). 
\end{lemma} 

\begin{proof} 
In general, let  
\(j\colon\namedright{X\vee Y}{}{X\times Y}\)  
be the inclusion of the wedge into the product. By definition, $\psi'$ is the composite 
\(\nameddright{\overline{M}}{\psi}{\overline{M}\vee(\bigvee_{i=1}^{d} S^{n-m})}{h\vee p}{S^{m}\vee S^{n-m}}\) 
and it was noted that $\psi'$ composed with the pinch map to $S^{m}$ is homotopic to $h$ while $\psi'$ 
composed  with the pinch map to $S^{n-m}$ is homotopic to $p'$.  Thus the composite 
\[\nameddright{\overline{M}}{\psi'}{S^{m}\vee S^{n-m}}{j}{S^{n}\times S^{n-m}}\] 
is the product map $h\times p'$. On the other hand, since $\sigma$ is a comultiplication 
it is a lift of the diagonal map 
\(\namedright{\overline{M}}{\Delta}{\overline{M}\times\overline{M}}\). 
The naturality of $j$ then implies that the composite 
\[\namedddright{\overline{M}}{\sigma}{\overline{M}\vee\overline{M}}{h\vee p'}{S^{m}\vee S^{n-m}}{j} 
     {S^{m}\times S^{n-m}}\] 
is homotopic to $h\times p'$. Thus if 
\[D\colon\namedright{\overline{M}}{}{S^{m}\vee S^{n-m}}\] 
is the difference $D=\psi'-(h\vee p')\circ\sigma$ then $j\circ D$ is null homotopic. By~\cite{G}, 
the homotopy fibre of $j$ is homotopy equivalent to $\Sigma\Omega S^{m}\wedge\Omega S^{n-m}$. 
Thus we obtain a lift 
\[\diagram 
     & \Sigma\Omega S^{m}\wedge\Omega S^{n-m}\dto \\ 
     \overline{M}\rto^-{D}\urto^-{\lambda} & S^{m}\vee S^{n-m} 
  \enddiagram\] 
for some map $\lambda$. Observe that $\overline{M}$ is $(n-m)$-dimensional while 
$\Sigma\Omega S^{m}\wedge\Omega S^{n-m}$ is $(n-2)$-connected. Since spaces 
are simply-connected we have $m\geq 2$, implying that $n-m\leq n-2$. Therefore $\lambda$ is null homotopic, 
implying that $D$ is null homotopic. Hence $\psi'\simeq(h\vee p')\circ\sigma$, as asserted. 
\end{proof} 

Next, recall from~(\ref{Mbarcofib}) that there is a homotopy cofibration  
\(\nameddright{S^{m}}{s}{\overline{M}}{a}{A}\). 
If $m < n-m$, the composite 
\(\nameddright{S^{m}}{s}{\overline{M}}{p'}{S^{n-m}}\) 
is null homotopic for dimension and connectivity reasons. If $m = n-m$, recall that the generators $x \in H^m(M)$ and $y \in H^{n-m}(M)$ have been chosen using Lemma~\ref{lem:distinctgens}, so it may be assumed that $y\in H^{m}(M)\backslash\mathbb{Z}\{x\}$. Therefore, \(\nameddright{S^{m}}{s}{\overline{M}}{p'}{S^{m}}\) 
is null homotopic as it has trivial image in cohomology by definition of $s$ and $p'$. 
Therefore $p'$ extends across $a$ to give the following. 

\begin{lemma} 
   \label{Afactor} 
   The map 
   \(\namedright{\overline{M}}{p'}{S^{n-m}}\) 
   factors as a composite 
   \(\nameddright{\overline{M}}{a}{A}{p''}{S^{n-m}}\) 
   for some map~$p''$.~$\qqed$ 
\end{lemma} 

Now things are put together to gain some control over the map $\gamma$ in~(\ref{psipb}). 

\begin{lemma} 
   \label{gammacontrol} 
   The map $\gamma$ in~(\ref{psipb}) factors as the composite 
   \(\lnameddright{E}{e'}{A\rtimes\Omega S^{m}}{p''\rtimes 1}{S^{n-m}\rtimes\Omega S^{m}}\), 
   where~$e'$ is the homotopy equivalence in Lemma~\ref{Etype}.  
\end{lemma} 

\begin{proof} 
First consider the diagram 
\[\diagram 
     \overline{M}\rto^-{\sigma}\dto^{\psi'} & \overline{M}\vee\overline{M}\rto^-{h\vee a}\dto^{h\vee p'} 
         & S^{m}\vee A\dto^{1\vee p''} \\ 
     S^{m}\vee S^{n-m}\rdouble & S^{m}\vee S^{n-m}\rdouble & S^{m}\vee S^{n-m}. 
  \enddiagram\] 
The left square homotopy commutes by Lemma~\ref{comatch} and the right square homotopy 
commutes by Lemma~\ref{Afactor}. The top row is the definition of the homotopy equivalence $e$ 
in Lemma~\ref{Mbarsplit}. Thus $\psi'\simeq (1\vee p'')\circ e$. This homotopy will let us factor the homotopy pullback defining $\gamma$ in~(\ref{psipb}). Consider 
the diagram
\[\begin{tikzcd}
	E & {A \rtimes \Omega S^m} & {S^{n-m} \rtimes \Omega S^m} \\
	{\overline{M}} & {S^m \vee A} & {S^m \vee S^{n-m}} \\
	{S^m} & {S^m} & {S^m}
	\arrow["{e'}", from=1-1, to=1-2]
	\arrow[from=1-1, to=2-1]
	\arrow["{p'' \rtimes 1}", from=1-2, to=1-3]
	\arrow[from=1-2, to=2-2]
	\arrow[from=1-3, to=2-3]
	\arrow["e", from=2-1, to=2-2]
	\arrow["h", from=2-1, to=3-1]
	\arrow["{1 \vee p''}", from=2-2, to=2-3]
	\arrow["{p_1}", from=2-2, to=3-2]
	\arrow["{p_1}", from=2-3, to=3-3]
	\arrow[Rightarrow, no head, from=3-1, to=3-2]
	\arrow[Rightarrow, no head, from=3-2, to=3-3]
\end{tikzcd}\] 
where the columns are homotopy fibrations. The map of homotopy fibrations 
between the left and middle columns is~(\ref{e'pb}) 
while the map of homotopy fibrations between the middle and right columns is due to the naturality 
of Lemma~\ref{Ganea}. Since the middle row is homotopic to $\psi'$, Lemma~\ref{gammaclass} implies that the top row is homotopic to $\gamma$, proving the lemma. 
\end{proof} 

One more step is needed. Consider the composite 
\[\lnameddright{A\rtimes\Omega S^{m}}{p''\rtimes 1}{S^{n-m}\rtimes\Omega S^{m}}{q} 
     {S^{n-m}\wedge\Omega S^{m}}.\] 
We will identify $q\circ(p''\rtimes 1)$ as being induced by a homotopy cofibration. To do so, we first identify~$p''$ as being induced by a homotopy cofibration. 

\begin{lemma} 
   \label{p''cofib} 
   There is a space $B$ and a map 
   $b:B \rightarrow A$ which induces 
    a homotopy cofibration 
   \(\nameddright{B}{b}{A}{p''}{S^{n-m}}\). 
\end{lemma} 

\begin{proof}  
If $m=n-m$, suppose $H_m(M)$ has rank $d$. Then $\overline{M} \simeq \bigvee_{i=1}^d S^m$, and by relabelling the wedge summands if necessary, the map $s:S^m \rightarrow \bigvee_{i=1}^d S^m$ can be taken to be the inclusion of the first wedge summand. Hence, $A \simeq \bigvee_{i=2}^d S^m$, and the map $p''$ can be chosen to be the pinch map onto the second wedge summand. Therefore, defining $B = \bigvee_{i=3}^d S^m$ and $b:\bigvee_{i=3}^d S^m \rightarrow \bigvee_{i=2}^d S^m$ as the inclusion, we obtain the asserted homotopy cofibration. 

Now suppose that $m < n-m$. By definition, the map 
\(\namedright{A}{p''}{S^{n-m}}\) 
factors the map 
\(\namedright{\overline{M}}{p'}{S^{n-m}}\), 
which induces an epimorphism in homology. Thus $p''$ also induces an epimorphism 
in homology. Define the space $F$ by the homotopy fibration 
\[\nameddright{F}{}{A}{p''}{S^{n-m}}.\] 
Taking the connecting map gives a homotopy fibration 
\(\nameddright{\Omega S^{n-m}}{}{F}{}{A}\). 
Since $A$ is $(m-1)$-connected for $m\geq 2$ and $\Omega S^{n-m}$ is $(n-m-2)$-connected, 
the Serre exact sequence implies that this homotopy fibration induces a long exact sequence 
in homology 
\begin{equation} 
  \label{Serreexact} 
  \namedddright{H_{n-m}(\Omega S^{n-m})}{}{H_{n-m}(F)}{}{H_{n-m}(A)}{} 
    {H_{n-m-1}(\Omega S^{n-m})}\longrightarrow\cdots 
 \end{equation} 
Notice that $H_{n-m}(\Omega S^{n-m})\cong 0$ unless $n-m=2$. In our case, 
$n-m> m\geq 2$, so \mbox{$H_{n-m}(\Omega S^{n-m})\cong 0$}. Thus if 
$F_{n-m}$ is the $(n-m)$-skeleton of $F$ then the exactness of~(\ref{Serreexact}) 
implies that there is a homotopy cofibration 
\(\nameddright{F_{n-m}}{}{A}{p''}{S^{n-m}}\). 
Taking $B=F_{n-m}$ and 
\(b\colon\namedright{B}{}{A}\) 
as
\(\namedright{F_{n-m}}{}{A}\), 
we obtain the assertion in the statement of the lemma.
\end{proof} 

Let 
\[i\colon\namedright{A}{}{A\rtimes\Omega S^{m}}\] 
be the inclusion into the first factor. Let $j$ be the composite 
\[j\colon\namedright{B\wedge\Omega S^{m}}{b\wedge 1}{A\wedge\Omega S^{m}} 
       \hookrightarrow A\rtimes\Omega S^{m}.\] 
Let 
\[i\perp j\colon\namedright{A\vee(B\wedge\Omega S^{m})}{}{A\rtimes\Omega S^{m}}\] 
be the wedge sum of $i$ and $j$.

\begin{lemma} 
   \label{qcofib} 
   There is a homotopy cofibration 
   \[\llnameddright{A\vee (B\wedge\Omega S^{m})}{i\perp j}{A\rtimes\Omega S^{m}} 
        {q\circ(p''\rtimes 1)}{S^{m-n}\wedge\Omega S^{m}}.\] 
\end{lemma} 

\begin{proof} 
The homotopy cofibration 
\(\nameddright{B}{b}{A}{p''}{S^{n-m}}\) 
implies there is a homotopy cofibration 
\[\nameddright{B\wedge\Omega S^{m}}{b\wedge 1}{A\wedge\Omega S^{m}}{p''\wedge 1} 
    {S^{n-m}\wedge\Omega S^{m}}.\] 
This in turn implies that there is a homotopy cofibration 
\[\lllnameddright{A\vee(B\wedge\Omega S^{m})}{1\vee(b\wedge 1)}{A\vee(A\wedge\Omega S^{m})}  
    {\ast\vee(p''\wedge 1)}{S^{n-m}\wedge\Omega S^{m}}.\] 
But as $A$ is a co-$H$-space, there is a homotopy equivalence 
$A\vee(A\wedge\Omega S^{m})\simeq A\rtimes\Omega S^{m}$, under which $1\vee(b\wedge 1)$ 
becomes $i\perp j$ and $\ast\vee(p''\wedge 1)$ becomes $(p''\wedge 1)\circ q$. The naturality of $q$ implies $(p'' \wedge 1) \circ q \simeq q \circ (p'' \rtimes 1)$. 
This gives the asserted homotopy cofibration. 
\end{proof} 

Finally, we will identify the homotopy type  of $E'$ and prove Theorem~\ref{main}.  

\begin{proposition} 
   \label{E'type} 
   There is a homotopy equivalence $E'\simeq A\vee (B\wedge\Omega S^{m})$. 
\end{proposition} 

\begin{proof} 
Consider the homotopy cofibration 
\[\nameddright{S^{n-1}\rtimes\Omega S^{m}}{\theta}{E}{}{E'}.\]  
By~(\ref{thetainv}), the composite 
\[\nameddright{S^{n-1}\rtimes\Omega S^{m}}{\theta}{E}{q\circ\gamma}{S^{n-m}\wedge\Omega S^{m}}\]
is a homotopy equivalence. By Lemma~\ref{gammacontrol}, $\gamma\simeq (p''\rtimes 1)\circ e'$. 
Thus the composite 
\[\lllnameddright{S^{n-1}\rtimes\Omega S^{m}}{e'\circ\theta}{A\rtimes\Omega S^{m}}{q\circ(p''\rtimes 1)} 
     {S^{n-m}\wedge\Omega S^{m}}\] 
is a homotopy equivalence. As the homotopy cofibre of $\theta$ is $E'$ and $e'$ is a homotopy  
equivalence, the homotopy cofibre of $e'\circ\theta$ is also $E'$. Therefore, using Lemma~\ref{qcofib}, 
we obtain a homotopy cofibration diagram 
\[\diagram 
       & A\vee (B\wedge\Omega S^{m})\rdouble\dto^{i\perp j} & A\vee (B\wedge\Omega S^{m})\dto \\ 
       S^{n-1}\rtimes\Omega S^{m}\rto^-{e'\circ\theta}\ddouble  & A\rtimes\Omega S^{m}\rto\dto^{q\circ(p''\rtimes 1)} 
            & E'\dto \\  
       S^{n-1}\rtimes\Omega S^{m}\rto^-{\simeq} & S^{n-m}\wedge\Omega S^{m}\rto & \ast. 
  \enddiagram\] 
The homotopy cofibration in the right column implies that the map 
\(\namedright{A\vee(B\wedge\Omega S^{m})}{}{E'}\) 
induces an isomorphism in homology, and is therefore a homotopy equivalence by Whitehead's Theorem 
since all spaces are simply-connected. 
\end{proof}  


\begin{proof}[Proof of Theorem~\ref{main}] 
Take 
\(\namedright{S^{m}}{s'}{M}\) 
and 
\(\namedright{M}{h'}{S^{m}}\) 
as the maps in the statement of the theorem.  By definition, $E'$ is the homotopy fibre of $h'$. 
By Proposition~\ref{E'type}, $E'\simeq A\vee(B\wedge\Omega S^{m})$. 
This proves the asserted homotopy fibration. Since $h'$ has a right 
homotopy inverse, the asserted homotopy equivalence for $\Omega M$ 
follows immediately. 
\end{proof} 

\begin{remark}
    \label{localMfib}
      There is a localised version of Theorem ~\ref{main} which will be used in Section ~\ref{sec:highyconnPDcomplex}. Let~$M$ be a $(k-1)$-connected Poincar\'e duality complex of dimension $n$, where $2 \leq k < n$. Let~$m$ be the least number such that $H_m(M)$ contains a $\mathbb{Z}$ summand, and suppose $m < n$. Let~$\Gamma$ be the set of primes appearing as $p$-torsion in $H_i(M)$ for $i < m$. Localised away from $\Gamma$, $M$ is an $(m-1)$-connected complex that satisfies Poincar\'e duality. In this case, if the hypotheses of Theorem ~\ref{main} hold after localisation away from $\Gamma$, then so do the conclusions. 
\end{remark} 

\begin{remark} 
\label{natremark} 
Theorem~\ref{main} satisfies a naturality property. Let  
$M$ and $N$ be two $(m-1)$-connected $n$-dimensional Poincar\'{e} 
Duality complexes where $2\leq m<n$, that $\overline{M}$ and $\overline{N}$ 
are co-$H$-spaces, and there are maps 
\(s_{M}\colon\namedright{S^{m}}{}{M}\) 
and 
\(s_{N}\colon\namedright{S^{m}}{}{N}\) 
having left homotopy inverses 
\(h_{M}\colon\namedright{M}{}{S^{m}}\) 
and 
\(h_{N}\colon\namedright{N}{}{S^{m}}\) 
respectively. Suppose that there is a map 
\(\alpha\colon\namedright{M}{}{N}\).   
Let 
\(\overline{\alpha}\colon\namedright{\overline{M}}{}{\overline{N}}\) 
be the restriction of $\alpha$ to $(n-1)$-skeletons. If  
(i) $\overline{\alpha}$ is a co-$H$-map, (ii) there is a homotopy 
commutative diagram 
\[\diagram 
     S^{m}\rto^-{s_{M}}\dto^{\beta} & M\rto^-{h_{M}}\dto^{\alpha} 
        & S^{m}\dto^{\beta} \\ 
     S^{m}\rto^-{s_{N}} & N\rto^-{h_{N}} & S^{m} 
   \enddiagram\] 
for some map $\beta$, and (iii) there is a homotopy commutative diagram 
of associated homotopy coactions 
\[\diagram 
     \overline{M}\rto^-{\psi'_{M}}\dto^{\overline{\alpha}} & \overline{M}\vee S^{n-m}\dto^{\overline{\alpha}\vee\lambda} \\ 
     \overline{N}\rto^-{\psi'_{N}} & \overline{N}\vee S^{n-m} 
  \enddiagram\]
for some map $\lambda$, then there is a homotopy fibration diagram 
\[\diagram 
     A_{M}\vee(B_{M}\wedge\Omega S^{m})\rto\dto^{a\vee(b\wedge\Omega\beta)} 
         & M\rto^-{h_{M}}\dto^{\alpha} & S^{m}\dto^{\beta} \\ 
     A_{N}\vee(B_{N}\wedge\Omega S^{m})\rto & N\rto^-{h_{N}} & S^{m} 
  \enddiagram\] 
for some maps $a$ and $b$, and there are correspondingly compatible loop space 
decompositions of $\Omega M$ and $\Omega N$. 

To explain why this is true, observe that the inclusions of $(n-1)$-skeletons 
leads to a homotopy cofibration diagram 
\[\diagram 
     S^{n-1}\rto^-{f_{M}}\dto^{d} 
       & \overline{M}\rto^-{i_{M}}\dto^{\overline{\alpha}} 
       & M\dto^{\alpha} \\ 
     S^{n-1}\rto^-{f_{N}} & \overline{N}\rto^-{i_{N}} & N 
  \enddiagram\] 
where $i_{M}$ and $i_{N}$ are the inclusions of the $(n-1)$-skeletons, 
$f_{M}$ and $f_{N}$ are the attaching maps for the $n$-cells, and 
$d$ is some map (of degree $d$). Note here that the right square 
clearly commutes by skeletal restriction, so it induces a map of 
homotopy fibres, and the simple-connectivity of $M$ and~$N$ implies 
by the Blakers-Massey Theorem that the map of fibres to total spaces 
coincides with a map of attaching maps in degrees~$\leq n-1$, giving 
the homotopy commutativity of the left square. This diagram of 
homotopy cofibrations, together with both the left and right squares 
in condition (ii), implies by \cite[Remark 2.7]{T1} that there is a homotopy 
cofibration diagram 
\[\diagram 
    S^{n-1}\rtimes\Omega S^{m}\rto^-{\theta_{M}}\dto^{d\rtimes\Omega\beta} 
        & E_{M}\rto\dto^{\epsilon} & E'_{M}\dto^{\epsilon'} \\  
    S^{n-1}\rtimes\Omega S^{m}\rto^-{\theta_{N}} & E_{N}\rto & E'_{N} 
  \enddiagram\] 
for some maps $\epsilon$ and $\epsilon'$. The homotopy commutative 
diagram of homotopy coactions in condition~(iii) implies that the 
construction of the left homotopy inverses of $\theta_{M}$ and $\theta_{N}$ 
are natural. These are used in combination with the homotopy decompositions 
of $E_{M}$ and $E_{N}$ in Lemma~\ref{Etype}. That Lemma is natural since 
$\overline{\alpha}$ is a co-$H$-map, giving a homotopy commutative diagram 
\[\diagram 
       \overline{M}\rto^-{e_{M}}\dto^{\overline{\alpha}} 
           & S^{m}\vee A_{M}\dto^{\beta\vee a} \\
       \overline{N}\rto^-{e_{N}} & S^{m}\vee A_{N} 
  \enddiagram\] 
where $e_{M}$ and $e_{N}$ are homotopy equivalences and $a$ is the map of 
homotopy cofibres induced by the homotopy commutativity of the left square 
in condition (ii). The pinch maps to $S^{m}$ on the right in this diagram 
are natural so there is an induced homotopy commutative diagram of fibres 
\[\diagram 
     E_{M}\rto^-{e'_{M}}\dto 
         & A_{M}\rtimes\Omega S^{m}\dto^{a\rtimes\Omega\beta} \\ 
     E_{N}\rto^-{e'_{N}} & A_{N}\rtimes\Omega S^{m} 
  \enddiagram\]
where $e'_{M}$ and $e'_{N}$ are homotopy equivalences. This together 
with the naturality of the left homotopy inverses for $\theta_{M}$ 
and $\theta_{N}$ imply that Lemma~\ref{gammacontrol} is natural. The 
space $B$ in Lemma~\ref{p''cofib} is constructed using the homotopy fibre of 
\(\namedright{A}{p''}{S^{n-m}}\), 
which satisfies a naturality property since this map is determined by 
\(\namedright{\overline{M}}{}{S^{n-m}}\) 
and the latter is natural because of condition~(iii). Thus all the 
ingredients in the statement and proof of Proposition~\ref{E'type} 
for the homotopy type of $E'$ are natural, and hence so is 
Theorem~\ref{main}. 
\end{remark}

\section{A refinement}
\label{sec:Refinement}

 In this section, the decomposition of $\Omega M$ in Theorem ~\ref{main} is refined when $\overline{M}$ is homotopy equivalent to a wedge of spheres and Moore spaces, in which case the spaces $A$ and $B$ can be explicitly identified.

As notation, let $\mathcal{W}$ be the collection of topological spaces that are homotopy equivalent to a finite type wedge of spheres and let $\mathcal{M}$ be the collection of topological spaces homotopy equivalent to a finite type wedge of spheres and Moore spaces. Note that $\mathcal{W}\subset\mathcal{M}$.
For a space $X$ and $n \geq 0$, let~$X^{\vee n}$ be the $n$-fold wedge sum of copies of $X$, where if $n = 0$ then $X^{\vee 0} = *$. 

Let $M$ be an $(m-1)$-connected $n$-dimensional Poincar\'{e} Duality complex. Separate the homology groups into torsion-free and torsion components: 
\[H_i(M) \cong \mathbb{Z}^{d_i} \oplus T_i\] 
where $d_i \geq 0$ and $T_i$ is a finite abelian group. Recall that 
$A$ is the homotopy cofibre of the map $S^m \xrightarrow{s} \overline{M}$, where $s$ has a left homotopy inverse, and by Lemma~\ref{p''cofib}, there is a homotopy cofibration 
\(\nameddright{B}{b}{A}{p''}{S^{n-m}}\). 

\begin{proposition}
    \label{prop:htpytypeAandB}
    Let $M$ be an $(m-1)$-connected, Poincar\'e duality complex of dimension $n$, where $2 \leq m < n$, and suppose that $\overline{M} \in \mathcal{M}$. Then there are homotopy equivalences \[A \simeq (S^{m})^{\vee d_m-1} \vee \bigvee\limits_{i=m+1}^{n-m}(S^{i})^{\vee d_i} \vee \bigvee\limits_{i=m}^{n-m-1} P^{i+1}(T_i)\] and \[B \simeq (S^{m})^{\vee d_m-1} \vee (S^{n-m})^{\vee d_{n-m}-1} \vee\bigvee\limits_{i=m+1}^{n-m-1}(S^{i})^{\vee d_i} \vee \bigvee\limits_{i=m}^{n-m-1} P^{i+1}(T_i).\]
\end{proposition}

\begin{proof}
   By Lemma ~\ref{Mbarsplit}, there is a homotopy equivalence \[\overline{M} \simeq S^m \vee A.\] This implies that $A$ retracts off $\overline{M}$. By \cite[Theorem 3.5]{St}, $\mathcal{M}$ is closed under retracts. Therefore, as $\overline{M}\in \mathcal{M}$ by hypothesis, we obtain $A\in \mathcal{M}$. The asserted  homotopy equivalence for $A$ therefore follows from the homology of $\overline{M}$.

By definition, $p''$ factors through 
\(\namedright{M}{p'}{S^{n-m}}\), 
which induces an epimorphism in homology. Therefore so does $p''$.
Thus, as $A\in \mathcal{M}$, 
there is a map $f\colon S^{n-m} \rightarrow A$ such that $p'' \circ f$ induces an isomorphism in homology. Thus $p''$ has a right homotopy inverse, implying that the homotopy cofibration 
    \(\nameddright{B}{b}{A}{p''}{S^{n-m}}\) 
    splits to give a homotopy equivalence 
    $A \simeq S^{n-m} \vee B$. In particular, $B$ retracts off $A$, implying that $B\in \mathcal{M}$. The asserted homotopy equivalence for $B$ therefore follows from the decomposition of $A$. 
\end{proof}

Applying Theorem~\ref{main}, we obtain the following.

\begin{corollary}
    \label{cor:MbarwedgeofsphereandMoore}
    Let $M$ be an $(m-1)$-connected, Poincar\'e duality complex of dimension $n$, where $2 \leq m < n$. Write $H_i(M) \cong \mathbb{Z}^{d_i} \oplus T_i$, where $d_i \geq 0$ and $T_i$ is a finite abelian group. Suppose that $\overline{M}\in \mathcal{M}$ and there is a map 
   \(\namedright{S^{m}}{}{M}\) 
   with a left homotopy inverse 
   \(\namedright{M}{h'}{S^{m}}\). Then there is a homotopy fibration 
   \[\nameddright{A\vee(B\wedge\Omega S^{m})}{}{M}{h'}{S^{m}},\] where \[A \simeq (S^{m})^{\vee d_m-1} \vee \bigvee\limits_{i=m+1}^{n-m}(S^{i})^{\vee d_i} \vee \bigvee\limits_{i=m}^{n-m-1} P^{i+1}(T_i)\] and \[B \simeq (S^{m})^{\vee d_m-1} \vee (S^{n-m})^{\vee d_{n-m}-1} \vee\bigvee\limits_{i=m+1}^{n-m-1}(S^{i})^{\vee d_i} \vee \bigvee\limits_{i=m}^{n-m-1} P^{i+1}(T_i),\]
and this homotopy fibration splits after looping to give a homotopy equivalence 
   \[\Omega M\simeq\Omega S^{m}\times\Omega(A\vee(B\wedge\Omega S^{m})).\eqno\qqed\] 
\end{corollary}


\section{Inert attaching maps and loop space homology} 
\label{sec:inert} 

Let $\mathcal{A}$ be the collection of Poincar\'{e} Duality complexes $M$ such that: 
\begin{enumerate} 
    \item $\overline{M}$ is a co-$H$ space; 
    \item if $M$ is $(m-1)$-connected with $2\leq m<\mbox{dim}\, M$ then there is a map \(\namedright{S^{m}}{s'}{M}\) that has a left homotopy inverse \(\namedright{M}{h'}{S^{m}}.\)
\end{enumerate} 
Theorem~\ref{main} applies to any $M\in\mathcal{A}$. In this section we show that 
the class $\mathcal{A}$ of Poincar\'{e} Duality complexes has two interesting properties: 
one is that $\Omega M$ retracts off $\Omega\overline{M}$, and the other is that 
$H_{\ast}(M;R)$ can be calculated as an algebra for appropriate rings $R$. 

Suppose that there is a homotopy cofibration 
\(\nameddright{A}{i}{X}{h}{Y}\). 
Following~\cite{T1}, the map $i$ is \emph{inert} if $\Omega h$ has a right homotopy inverse. 
This definition is inspired from rational homotopy theory~\cite{HL}, where in a homotopy cofibration 
\(\nameddright{S^{n-1}}{f}{X}{h}{X\cup e^{n}}\) 
the attaching map $f$ is inert if $\Omega h$ has a right homotopy inverse. (Rationally, the 
inert property tends to be equivalently described in terms of the associated rational homotopy 
Lie algebra.) The idea is that an inert cell attachment kills off homotopy groups 
of $X$ but does not introduce new ones in $X\cup e^{n}$. 

In~\cite{BT2} it was shown that if $\Omega i$ has a right homotopy inverse then there is 
a homotopy fibration 
\[\nameddright{A\rtimes\Omega M}{}{\overline{M}}{i}{M}\] 
that splits after looping to give a homotopy equivalence 
\[\Omega\overline{M}\simeq\Omega M\times\Omega(A\rtimes\Omega M).\] 
In~\cite{T2} it was shown that if $M$ is an $(m-1)$-connected $n$-dimensional Poincar\'{e} 
Duality complex and there is a map 
\(\namedright{S^{m}}{}{M}\) 
with a left homotopy inverse, then in the homotopy cofibration 
\(\nameddright{S^{n-1}}{f}{\overline{M}}{i}{M}\), 
the attaching map for the $n$-cell of $M$ is inert. The hypothesis on the map 
\(\namedright{S^{m}}{}{M}\) 
is exactly the second condition for being in the class $\mathcal{A}$. Thus we obtain the following 
re-statement of Theorem~\ref{introinert}.  

\begin{theorem} 
   \label{Minert} 
   Let $M$ be an $(m-1)$-connected $n$-dimensional Poincar\'{e} Duality complex. 
   If $M\in\mathcal{A}$ then the attaching map for the $n$-cell of $M$ is inert. Consequently, 
   there is a homotopy fibration 
   \[\nameddright{S^{n-1}\rtimes\Omega M}{}{\overline{M}}{i}{M}\] 
   that splits after looping to give a homotopy equivalence 
   \[\Omega\overline{M}\simeq\Omega M\times\Omega(S^{n-1}\rtimes\Omega M).\] 
\end{theorem} 
\vspace{-1cm}~$\qqed$\bigskip 

There is a useful homological consequence of an inert map, derived from the following more general 
statement proved in~\cite[Proposition 10.1]{T1}.

\begin{proposition}
    \label{prop:loophominert}
    Suppose there is a homotopy cofibration \[\Sigma A \xrightarrow{f} \Sigma X \xrightarrow{h} Y\] where $\Omega h$ has a right homotopy inverse. Let $\widetilde{f}:A \rightarrow \Omega \Sigma X$ be the adjoint of $f$ and let $R$ be a commutative ring with unit such that $H_*(\Sigma X;R)$ is a free-$R$ module. Then there is an algebra isomorphism \[H_*(\Omega Y;R) \cong T(\widetilde{H}_*(X);R)/(\mathrm{Im}(\widetilde{f}_*)),\] where $(\mathrm{Im}(\widetilde{f}_*))$ is the two sided ideal generated by $\mathrm{Im}(\widetilde{f}_*)$. Moreover, if $X$ is a suspension then this is an isomorphism of Hopf algebras. $\qqed$
\end{proposition} 

The proof of Proposition~\ref{prop:loophominert} uses the Bott-Samelson Theorem, 
which says there is an algebra isomorphism 
$H_{\ast}(\Omega\Sigma X;R)\cong T(\rhlgy{X;R})$ 
that is an isomorphism of Hopf algebras if $X$ is a suspension. Berstein~\cite{Ber} 
generalised this to looped co-$H$-spaces: if $C$ is a co-$H$-space then there 
is an algebra isomorphism 
$H_{\ast}(\Omega C;R)\cong T(\Sigma^{-1}\rhlgy{C;R})$, 
where $\Sigma^{-1}\rhlgy{C;R}$ is $\rhlgy{C;R}$ shifted down one degree, and 
this is an isomorphism of Hopf algebras if $C$ is the suspension of a co-$H$-space. 
The argument proving Proposition~\ref{prop:loophominert} in~\cite{T1} goes through 
verbatim in the more general case of a looped co-$H$-space. Thus we obtain the 
following re-statement of Theorem~\ref{introloophlgy}.   

\begin{theorem} 
   \label{Minertloophlgy} 
   Let $M$ be an $(m-1)$-connected $n$-dimensional Poincar\'{e} Duality complex. 
   If $M\in\mathcal{A}$ then for any commutative ring $R$ with unit such that 
   $H_{\ast}(\overline{M};R)$ is a free-$R$-module, there is an algebra isomorphism 
   \[H_*(\Omega M;R) \cong 
       T(\Sigma^{-1}\widetilde{H}_*(\overline{M});R)/(\mathrm{Im}(\widetilde{f}_*)),\] 
   where $(\mathrm{Im}(\widetilde{f}_*))$ is the two sided ideal generated by 
   $\mathrm{Im}(\widetilde{f}_*)$. Moreover, if $\overline{M}$ is the suspension of a 
   co-$H$-space then this is an isomorphism of Hopf algebras.~$\qqed$ 
\end{theorem} 

\begin{proof} 
There is a homotopy cofibration 
\(\nameddright{S^{n-1}}{f}{\overline{M}}{}{M}\) 
where $f$ attaches the $n$-cell to $M$. By hypothesis, $M\in\mathcal{A}$, so 
$\overline{M}$ is a co-$H$-space. Now apply Proposition~\ref{prop:loophominert}. 
\end{proof}

\section{Initial examples} 
\label{sec:examples} 

The next two sections build up an array of examples to which Theorem~\ref{main} 
or its refinement in Corollary~\ref{cor:MbarwedgeofsphereandMoore} apply, as well 
as the two properties in Theorems~\ref{Minert} and~\ref{Minertloophlgy}. This 
section considers some initial examples that will then feed into the operations on 
Poincar\'{e} Duality complexes considered in the next section. 
\medskip 

\noindent 
\textbf{$(n-1)$-connected $(2n)$-dimensional Poincar\'{e} Duality complexes with $n \notin \{2,4,8\}$}. 
Let $M$ be an $(n-1)$-connected $(2n)$-dimensional Poincar\'{e} Duality complex 
for $n\geq 2$. Then $\overline{M}$ is homotopy equivalent to a wedge of copies of $S^{n}$. In particular, $\overline{M}$ is a co-$H$-space. If the rank of 
$H_{n}(M)$ is at least $2$ and $n \notin \{2,4,8\}$, then in~\cite{BT2} it is shown that there is a map 
\(\namedright{S^{n}}{}{M}\) 
with a left homotopy inverse. Thus $M$ satisfies the hypotheses of Corollary ~\ref{cor:MbarwedgeofsphereandMoore}. 
\medskip 

\noindent 
\textbf{$(n-1)$-connected $(2n+1)$-dimensional Poincar\'{e} Duality complexes}. 
Let $M$ be an $(n-1)$-connected $(2n+1)$-dimensional Poincar\'{e} Duality complex 
for $n\geq 2$. Then $\overline{M}$ is homotopy equivalent to a wedge of copies of $S^{n}$, $S^{n+1}$ 
and $(n+1)$-dimensional Moore spaces. In particular, $\overline{M}$ is a co-$H$-space. If the rank of 
$H_{n}(M)$ is at least $1$, then in~\cite{BT2} it is shown that there is a map 
\(\namedright{S^{n}}{}{M}\) 
with a left homotopy inverse. Thus $M$ satisfies the hypotheses of Corollary ~\ref{cor:MbarwedgeofsphereandMoore}. 
\medskip 

\noindent 
\textbf{Connected sums of products of two spheres}. 
Let $M$ be a connected sum of products of two spheres, 
$M=\displaystyle\conn_{i=1}^{d} (S^{m_{i}}\times S^{n-m_{i}})$. 
Then $\overline{M}\simeq\bigvee_{i=1}^{d} (S^{m_{i}}\vee S^{n-m_{i}})$, 
so $\overline{M}\in\mathcal{W}$. 
Let $m$ be the minimum of $\{m_{i},n-m_{i}\}_{i=1}^{d}$. Then there is a map 
\(\namedright{S^{m}}{s'}{M}\) 
that first includes $S^{m}$ into the wedge of spheres in $\overline{M}$ and then includes into $M$. 
{Within $M$, collapsing out all the spheres in $\overline{M}$ except the pair $S^{m}\vee S^{n-m}$ produces a map 
\(\namedright{M}{}{S^{m}\times S^{n-m}}\). 
Composing with the projection to $S^{m}$ then gives a left homotopy inverse for~$s'$. Thus $M$ satisfies the hypotheses of Corollary ~\ref{cor:MbarwedgeofsphereandMoore} when $2\leq m < n$.
\medskip

\noindent 
\textbf{Connected sums of $S^{n-m}$-bundles over $S^{m}$}.  
We begin with a lemma. 

\begin{lemma} 
   \label{spherebundleA} 
   Let $M$ be an $S^{n-m}$-bundle over $S^{m}$ with $2 \leq m < n-1$. 
   Then $M\in\mathcal{A}$. 
\end{lemma} 

\begin{proof} 
There is a homotopy fibration 
\(\nameddright{S^{n-m}}{}{M}{}{S^{m}}\)  
and $M$ is a Poincar\'{e} Duality complex. Since $m < n$, the Hurewicz isomorphism 
implies that there is a map 
\(\namedright{S^{m}}{}{M}\) 
that is a right homotopy inverse for the bundle map 
\(\namedright{M}{}{S^{m}}\). 
Also, $\overline{M}\simeq S^{m}\vee S^{n-m}$, so $\overline{M}$ is a co-$H$-space. 
Thus $M\in\mathcal{A}$. 
\end{proof} 

One possibility for $M$ is the product $S^{m}\times S^{n-m}$  
but there are also nontrivial bundles (for example, see~\cite{JW}). In these cases, the 
attaching map for the $n$-cell of $M$ is a map 
\(\namedright{S^{n-1}}{}{S^{m}\vee S^{n-m}}\) 
that is nontrivial when pinched to $S^{n-m}$. 

More generally, let $M=N_{1}\conn\cdots\conn N_{d}$ where each $N_{i}$ is an 
$n$-dimensional $S^{n-m_{i}}$-bundle over $S^{m_{i}}$ with $m_{i} < n-1$. Then 
$\overline{M}\simeq\bigvee_{i=1}^{d} (S^{m_{i}}\vee S^{n-m_{i}})$ and if $m$ is the 
minimum of $\{m_{i}\}_{i=1}^{d}$ then there is a map 
\(\namedright{S^{m}}{}{M}\) 
that has a left homotopy inverse given by the composite 
\(\nameddright{M}{c}{N_{i_{0}}}{q}{S^{m}}\), 
where $i_{0}$ is an index with $m_{i_{0}}$ achieving the minimum $m$, 
the map $c$ collapses the connected sum to the factor $N_{i_{0}}$, and $q$ 
is the bundle map. Thus $M\in\mathcal{A}$.

\begin{remark} 
A decomposition for $\Omega M$ when $M$ is $(n-1)$-connected, $(2n+1)$-dimensional and the rank of $H_{n}(M)$ is at least $1$ is known~\cite{Bas,BT2}, a decomposition for $\Omega M$ when $M$ is the connected sum of products of two spheres is also known~\cite{BT1}, and 
a decomposition for $\Omega M$ when $M$ is a connected sum of $S^{n-m}$-bundles 
over $S^{m}$ with $m<n-1$ can be deduced from a decomposition in~\cite{T1} of the 
loops on a connected sum where one factor is inert. However, the first two decompositions 
arise in a different context than the latter. Theorem~\ref{main} deals with all three uniformly 
and, in fact, gives a more refined decomposition in the case of the loops on a connected 
sum of sphere-bundles over spheres.
\end{remark} 

A completely new example is the following.
\medskip 

\noindent 
\textbf{Neighbourly moment-angle manifolds}.
Let $K$ be a simplicial complex on  the vertex set \mbox{$[m]=\{1,\ldots,m\}$}. For $\sigma\in K$, let 
\[(D^{2},S^{1})^{\sigma}=\prod_{i=1}^{m} Y_{i}\] 
where $Y_i =  D^2$ if $i \in \sigma$ and $Y_i = S^1$ if $i \notin \sigma$. The \textit{moment-angle complex} $\zk$ associated to $K$ is defined as 
\[\zk = \bigcup\limits_{\sigma \in K} (D^{2},S^{1})^{\sigma}.\] 
If $K$ is a triangulation of $S^{n}$ then $\zk$ is a manifold of dimension $m+n+1$ \cite[Theorem 4.1.4]{BP}. 

A simplicial complex is \emph{$k$-neighbourly} if any set of $k+1$ vertices spans a simplex. If $K$ is a $k$-neighbourly simplicial complex then $\zk$ is $(2k+2)$-connected \cite[Proposition 4.3.5 (b)]{BP}. A triangulation $K$ of $S^{2n+1}$ is called \emph{neighbourly} if $K$ is $n$-neighbourly.
\begin{proposition}
    \label{prop:neighbourlysphere}
    Let $K \neq \partial{\Delta^{2n+2}}$ be a neighbourly triangulation of $S^{2n+1}$ on $[m]$. Then there is a homotopy fibration 
   \[\nameddright{A\vee(B\wedge\Omega S^{2n+3})}{}{\zk}{h'}{S^{2n+3}}\] 
   where $A$ and $B$ are as in Theorem~\ref{main}, and this homotopy fibration splits after looping to give a homotopy equivalence 
   \[\Omega \zk\simeq\Omega S^{2n+3}\times\Omega(A\vee(B\wedge\Omega S^{2n+3})).\] Further, if $n=1$ then $A$ and $B$ are wedges of spheres that can be explicitly enumerated as in Corollary ~\ref{cor:MbarwedgeofsphereandMoore}.  
\end{proposition} 

\begin{proof} 
 We will check that the hypotheses of Theorem~\ref{main} hold. 

 First, we claim that $K$ contains a minimal missing face of dimension $n+1$. Suppose not. As~$K$ is a neighbourly triangulation of $S^{2n+1}$ it is $n$-neighbourly. Having no minimal missing face of dimension $n+1$ implies that $K$ is $(n+1)$-neighbourly. But then ~\cite[Theorem 1.6]{IK} implies that $\zk$ is a co-H space. However, a moment-angle manifold $\zk$ where $K$ is a triangulation of a sphere is a co-$H$ space if and only if $K$ is the boundary of a simplex (c.f \cite[Lemma 6.3]{ST} for example), contradicting the hypothesis that $K\neq\partial\Delta^{2n+2}$. Thus $K$ has a minimal missing face of dimension~$n+1$. Therefore, by~\cite[Example 5.4]{T2} for example, there is a map $\zk \rightarrow S^{2n+3}$ that has a right homotopy inverse. 
 
Next, if $K$ is a neighbourly triangulation of $S^{2n+1}$ on $[m]$ then the dimension of $\zk$ 
 is $m+2n+2$ and its connectivity is $2n+2$. Since $2n+1 \geq 3$, $m \geq 4$ and so, $2n+3 < m+2n+2$.

Finally, in \cite[Theorem 6.6]{ST}, it was shown that if $K$ is a neighbourly sphere then there is a homotopy equivalence \[\overline{\zk} \simeq \bigvee \limits_{I \notin K,I \neq [m]} \Sigma^{1+|I|} |K_I|.\] In particular, $\overline{\zk}$ is a  suspension, so it is a co-$H$ space. Moreover, if $K$ is a triangulation of $S^{3}$ then $\overline{\mathcal{Z}_K}$ is homotopy equivalent to a wedge of spheres~\cite[Lemma 6.8]{ST}. 
 
 Hence all the hypotheses of Theorem ~\ref{main} hold, and applying it gives 
 the statement of the proposition.
\end{proof}

\begin{remark} 
The homotopy decomposition for $\Omega\zk$ in Proposition~\ref{prop:neighbourlysphere} when $n=1$ improves on~\cite[Theorem 1.2]{ST}, which showed that if $K$ is a neighbourly triangulation of $S^{3}$ then $\Omega\zk$ is homotopy equivalent to a product of spheres and loops on spheres, but the number and dimensions of the spheres involved were not explicitly enumerated. 
\end{remark}

\section{Connected sums and gyrations}
\label{sec:connsumgyration}

In this section we show that if $M\in\mathcal{A}$ then the connected sum and gyration operations produce new members of $\mathcal{A}$. 
\medskip

\noindent 
\textbf{Connected sums}.
Let $M$ and $N$ be simply-connected $n$-dimensional Poincar\'{e} Duality complexes. There are 
homotopy cofibrations 
\[\nameddright{S^{n-1}}{f}{\overline{M}}{i}{M}\qquad\qquad\nameddright{S^{n-1}}{g}{\overline{N}}{j}{N}\] 
where $f$ and $g$ are the attaching maps for the $n$-cells and $i$ and $j$ are skeletal inclusions. 
The connected sum $M\conn N$ is formed by removing the interior of an $n$-disc $D^{n}$ from each of $M$ 
and~$N$ and gluing $M\backslash (D^{n})^{\circ}$ and $N\backslash (D^{n})^{\circ}$ together 
along their boundaries. Notice that $M\conn N$ is a simply-connected $n$-dimensional Poincar\'{e} 
Duality complex. A topological description is as follows. Observe that $M\backslash (D^{n})^{\circ}$ 
and $N\backslash (D^{n})^{\circ}$ are homotopy equivalent to $\overline{M}$ and $\overline{N}$ respectively. 
Observe also that the $(n-1)$-skeleton of $M\conn N$ is homotopy equivalent to $\overline{M}\vee\overline{N}$. 
The attaching map for the $n$-cell of $M\conn N$ is given by the composite 
\(f+g\colon\nameddright{S^{n-1}}{\sigma}{S^{n-1}\vee S^{n-1}}{f\vee g}{\overline{M}\vee\overline{N}}\), 
where $\sigma$ is the comultiplication. Thus there is a homotopy cofibration 
\begin{equation} 
  \label{connsumcofib} 
  \nameddright{S^{n-1}}{f+g}{\overline{M}\vee\overline{N}}{}{M\conn N}. 
\end{equation}  

\begin{proposition} 
   \label{connsumcoH} 
   Let $M$ and $N$ be simply-connected $n$-dimensional Poincar\'{e} Duality complexes such 
   that $M\in\mathcal{A}$, $\overline{N}$ is a co-$H$-space, and the connectivity of $M$ is less than 
   or equal to the connectivity of $N$. Then $M\conn N\in\mathcal{A}$. 
\end{proposition} 

\begin{proof} 
Suppose that $M$ is $(m-1)$-connected. The homotopy cofibration~(\ref{connsumcofib}) implies 
$\overline{M\conn N}\simeq\overline{M}\vee\overline{N}$. In particular, as the connectivity of $M$ is 
less than or equal to that of $N$, $M\conn N$ is $(m-1)$-connected. We check that $M\conn N$ 
satisfies both conditions required to be in the class $\mathcal{A}$. 

First, $\overline{M}$ is a co-$H$-space since $M\in\mathcal{A}$ and, by hypothesis, $\overline{N}$ 
is a co-$H$-space. Thus $\overline{M}\vee\overline{N}$ is a co-$H$-space.  

Second, as $M\in\mathcal{A}$ there is a map 
\(\namedright{S^{m}}{s'}{M}\) 
with a left homotopy inverse 
\(\namedright{M}{h'}{S^{m}}\). 
Since $M$ is $n$-dimensional and by assumption $m<n$, 
$s'$ factors through the $(n-1)$-skeleton to give a map 
\(\namedright{S^{m}}{s}{\overline{M}}\). 
Collapsing $\overline{N}$ to a point inside $M\conn N$ gives a map 
\(\namedright{M\conn N}{}{M}\) 
whose restriction to~$\overline{M}$ is the inclusion of $\overline{M}$ into $M$.
Thus the composite 
\(\namedright{S^{m}}{s}{\overline{M}}\hookrightarrow\namedright{\overline{M}\vee\overline{N}} 
     {}{M\conn N}\) 
has a left homotopy inverse given  by 
\(\nameddright{M\conn N}{}{M}{h'}{S^{m}}\). 
\end{proof} 

It is notable in Proposition~\ref{connsumcoH} that $N$ does not have to satisfy 
condition~(2) for being in $\mathcal{A}$. This lets us inflate the examples to which 
Theorem~\ref{main} applies.  

\begin{corollary} 
   \label{connsumcoHcor} 
   Let $N$ be any $n$-dimensional Poincar\'{e} Duality complex such that $\overline{N}$ 
   is a co-$H$-space. Then for any $2\leq m < n-1$ we have 
   $(S^{m}\times S^{n-m})\conn N\in\mathcal{A}$. 
\end{corollary} 
   
\begin{proof}    
Observe that the product $S^{m}\times S^{n-m}\in\mathcal{A}$ since 
$\overline{S^{m}\times S^{n-m}}\simeq S^{m}\vee S^{n-m}$ is a co-$H$-space 
and the inclusion 
\(\namedright{S^{m}}{}{S^{m}\times S^{n-m}}\) 
has a left homotopy inverse given by the projection  
\(\namedright{S^{m}\times S^{n-m}}{}{S^{m}}\). 
Thus Proposition~\ref{connsumcoH} implies that $(S^{m}\times S^{n-m})\conn N\in\mathcal{A}$. 
\end{proof}  

Similarly, as any $S^{n-m}$-bundle over $S^{m}$ with $m < n-1$ 
is in $\mathcal{A}$ by Lemma~\ref{spherebundleA}, we also have the following. 

\begin{corollary} 
   \label{connsumcoHcor2} 
   Let $N$ be any $n$-dimensional Poincar\'{e} Duality complex such that $\overline{N}$ 
   is a co-$H$-space. If $M$ is an $S^{n-m}$-bundle over $S^{m}$ with $m< n-1$ then 
   $M\conn N\in\mathcal{A}$.~$\qqed$  
\end{corollary} 

We give an interesting example for each of these corollaries. 


\begin{example} 
Let $W$ be the Wu manifold, $W=SU(3)/SO(3)$. Then $W$ is a simply-connected $5$-dimensional 
Poincar\'{e} Duality complex and $\overline{W}$ is the $3$-dimensional mod-$2$ Moore 
space $P^{3}(2)$. In particular, $\overline{W}$ is a co-$H$-space. Thus 
Corollary~\ref{connsumcoHcor} implies that $(S^{2}\times S^{3})\conn W\in\mathcal{A}$. Further, as 
$\overline{(S^{2}\times S^{3})\conn W}\simeq S^{2}\vee S^{3}\vee P^{3}(2)$ is in $\mathcal{M}$, 
Corollary~\ref{cor:MbarwedgeofsphereandMoore} implies there is a homotopy fibration 
\[\nameddright{A\vee(B\wedge\Omega S^{2})}{}{(S^{2}\times S^{3})\conn W}{}{S^{2}}\] 
where $A\simeq S^{3}\vee P^{3}(2)$ and $B\simeq P^{3}(2)$, and this homotopy fibration 
splits after looping. 

Similarly, there is a nontrivial $S^{3}$-bundle over $S^{2}$ denoted $S^{2}\widetilde{\times} S^{3}$. 
Corollary~\ref{connsumcoHcor2} implies that 
$(S^{2}\widetilde{\times} S^{3})\conn W\in\mathcal{A}$, and as 
$\overline{(S^{2}\widetilde{\times} S^{3})\conn W}\simeq S^{2}\vee S^{3}\vee P^{3}(2)$ 
is in $\mathcal{M}$, Corollary~\ref{cor:MbarwedgeofsphereandMoore} implies there is 
a homotopy fibration 
\[\nameddright{A\vee(B\wedge\Omega S^{2})}{}{(S^{2}\widetilde{\times }S^{3})\conn W}{}{S^{2}}\] 
where $A\simeq S^{3}\vee P^{3}(2)$ and $B\simeq P^{3}(2)$, and this homotopy fibration 
splits after looping. Consequently 
$\Omega((S^{2}\times S^{3})\conn W)\simeq\Omega((S^{2}\widetilde{\times} S^{3})\conn W)$. 
\end{example}

\noindent 
\textbf{Gyrations}. 
Let 
\(\tau\colon\namedright{S^{k-1}}{}{SO(n)}\) 
be a map. Using the standard action of $SO(n)$ on $S^{n-1}$, define the map 
\(\vartheta\colon\namedright{S^{n-1}\times S^{k-1}}{}{S^{n-1}\times S^{k-1}}\) 
by $\vartheta(a,t)=(\tau(t)\cdot a,t)$. Recall that 
\(\namedright{S^{n-1}}{f}{\overline{M}}\) 
is the attaching map for the top cell of $M$, and let 
\(i\colon\namedright{S^{k-1}}{}{D^{k}}\) 
be the standard inclusion. For any integer $k\geq 2$, the \emph{twisted gyration} 
$\mathcal{G}^{k}_{\tau}(M)$ is defined by the pushout 
\begin{equation} 
  \label{gyrationpo} 
  \diagram 
     S^{n-1}\times S^{k-1}\rto^-{1\times i}\dto^{(f\times 1)\circ\vartheta} & S^{n-1}\times D^{k}\dto \\ 
     \overline{M}\times S^{k-1}\rto & \mathcal{G}^{k}_{\tau}(M).  
  \enddiagram  
\end{equation} 
The twisted gyration is an $(n,k-1)$-surgery, implying that $\mathcal{G}^{k}_{\tau}(M)$ is an 
$(n+k-1)$-dimensional Poincar\'{e} Duality complex. If $\tau$ is the trivial map, 
denote the associated \emph{non-twisted gyration} by $\mathcal{G}^{k}_{0}(M)$. The 
non-twisted gyration plays an important role in determining the diffeomorphism types of 
certain moment-angle manifolds in toric topology~\cite{GLdM} while the twisted gyration 
plays an important role in classifying circle bundles over manifolds~\cite{D}. 

We will show that the gyration for any choice of twisting preserves the property of being in $\mathcal{A}$. This makes use of the identification of $\overline{\mathcal{G}^{k}_{\tau}(M)}$ 
by Basu and Ghosh~\cite[Proposition 6.9]{BG}.  

\begin{lemma} 
   \label{BasuGhosh} 
   For all $k\geq 1$ and all $\tau$, there is a homotopy equivalence 
   $\overline{\mathcal{G}^{k}_{\tau}(M)}\simeq\overline{M}\rtimes S^{k-1}$.~$\qqed$ 
\end{lemma} 

It will be convenient for applications to consider the two properties of being in $\mathcal{A}$ separately. 

\begin{lemma} 
   \label{pregyrationcoH} 
   Let $M$ be a simply-connected $n$-dimensional Poincar\'{e} Duality complex such that 
   $\overline{M}$ is a co-$H$-space. Then $\overline{\mathcal{G}^{k}_{\tau}(M)}$ is a co-$H$-space 
   for any $k\geq 2$ and any $\tau$. 
\end{lemma} 

\begin{proof} 
By Lemma~\ref{BasuGhosh}, there is a homotopy equivalence 
$\overline{\mathcal{G}^{k}_{\tau}(M)}\simeq\overline{M}\rtimes S^{k-1}$. By hypothesis, 
$\overline{M}$ is a co-$H$-space. It is well known that if $A$ is a co-$H$-space then 
for any space $B$, the half-smash $A\rtimes B$ is also a co-$H$-space. Thus 
$\overline{M}\rtimes S^{k-1}$ and hence $\overline{\mathcal{G}^{k}_{\tau}(M)}$ is a co-$H$-space. 
\end{proof} 

\begin{lemma} 
   \label{gyrationspheremap} 
   Let $M$ be an $(m-1)$-connected, $n$-dimensional Poincar\'{e} Duality complex with $m < n$, such that 
   there is a map $S^m \rightarrow M$ with a left homotopy inverse. Then  
   for any $k\geq 2$ and any $\tau$, there is a map $S^m \rightarrow \mathcal{G}_{\tau}^k(M)$ with a left homotopy inverse. 
\end{lemma} 

\begin{proof} 
Consider the homotopy cofibration \[S^{n-1} \xrightarrow{f} \overline{M} \xrightarrow{i} M,\] where $f$ is the attaching map of the top cell. By hypothesis, there is a map $s:S^m \rightarrow M$ which has a left homotopy inverse $r:M \rightarrow S^m$. Since $m < n$, the map $s$ factors through the $(n-1)$-skeleton $\overline{M}$ via a map $s':S^m \rightarrow \overline{M}$. Further, the composite \[r':\overline{M} \xrightarrow{i} M \xrightarrow{r} S^m\] is a left homotopy inverse for $s'$. Note that the composite $r' \circ f$ is null homotopic. 

Now consider the homotopy fibration \[S^{n+k-2} \xrightarrow{\phi_\tau} \overline{M} \rtimes S^{k-1} \rightarrow \mathcal{G}_{\tau}^k(M),\] where $\phi_\tau$ is the attaching map of the top cell. In \cite[Lemma 3.2]{CT} the map $\phi_\tau$ was identified as the composite \[\phi_\tau:S^{n+k-2} \xrightarrow{j} S^{n-1} \rtimes S^{k-1} \xrightarrow{t'}S^{n-1} \rtimes S^{k-1} \xrightarrow{f \rtimes 1} \overline{M} \rtimes S^{k-1},\] where $j$ is the restriction of the homotopy equivalence $S^{n-1} \vee S^{n+k-2} \rightarrow S^{n-1} \rtimes S^{k-1}$ to $S^{n+k-2}$, and $t'$ is a map depending on the choice of $\tau$. Since $r' \circ f$ is null homotopic, the composite \[S^{n-1} \rtimes S^{k-1} \xrightarrow{f \rtimes 1} \overline{M} \rtimes S^{k-1} \xrightarrow{r' \rtimes 1} S^m \rtimes S^{k-1} \xrightarrow{\pi} S^m,\] where $\pi$ is the projection, is null homotopic. Hence, $\pi \circ (r' \rtimes 1)\circ \phi_\tau$ is null homotopic, implying that there is an extension of $\pi \circ (r' \rtimes 1)$ to a map $r'':\mathcal{G}_{\tau}^k(M) \rightarrow S^m$. The right homotopy inverses for $\pi$ and~$r'$ imply that $r''$ has a right homotopy inverse. Hence there is a map $s'':S^m \rightarrow \mathcal{G}_\tau^k(M)$ with a left homotopy inverse.
\end{proof} 

As a consequence, combining \cite[Theorem 1.1]{T2} and Lemma~\ref{gyrationspheremap}, we significantly extend the known 
examples of gyrations satisfying the inertness property proved in~\cite{Hu}.

\begin{theorem}
   Let $M$ be an $(m-1)$-connected, $n$-dimensional Poincar\'{e} Duality complex such that 
   there is a map $S^m \rightarrow M$ with a left homotopy inverse. Then  
   for any $k\geq 2$ and any $\tau$, the attaching map of the top cell of $\mathcal{G}_{\tau}^k(M)$ is inert.  $\qqed$
\end{theorem}

Combining Lemmas~\ref{pregyrationcoH} and~\ref{gyrationspheremap}, while noting that $M$ and $\mathcal{G}^{k}_{\tau}(M)$ have the same connectivity by Lemma~\ref{BasuGhosh}, we obtain the following.

\begin{theorem} 
   \label{gyrationcoH} 
   Let $M$ be a simply-connected $n$-dimensional Poincar\'{e} Duality complex such that 
   \mbox{$M\in\mathcal{A}$}. Then for any $k\geq 2$ and any $\tau$, the gyration 
   $\mathcal{G}^{k}_{\tau}(M)$ is a simply-connected $(n+k-1)$-dimensional Poincar\'{e} 
   Duality complex with $\mathcal{G}^{k}_{\tau}(M)\in\mathcal{A}$.~$\qqed$ 
\end{theorem} 

\begin{remark} 
In~\cite{HT2}, it is shown that for any simply-connected Poincar\'{e} Duality complex $M$ 
there is a homotopy equivalence 
$\Omega\mathcal{G}^{k}_{0}(M)\simeq\Omega\overline{M}\times\Omega\Sigma^{k} F$, 
where $F$ is the homotopy fibre of the attaching map 
\(\namedright{S^{n-1}}{}{\overline{M}}\) 
for the $n$-cell of $M$. The space $F$ may not be explicitly described. It is also shown that the same decomposition holds for $\Omega\mathcal{G}^{k}_{\tau}(M)$ with nontrivial $\tau$ after localisation away from a finite set of primes depending on the image of the $J$-homomorphism. However, 
if $M\in\mathcal{A}$ then there are significant improvements. Theorem~\ref{gyrationcoH} implies that $\mathcal{G}^{k}_{\tau}(M)\in\mathcal{A}$, and therefore Theorem~\ref{main} can be applied. This gives an integral homotopy decomposition for $\Omega\mathcal{G}^{k}_{\tau}(M)$ for all $\tau$ and one in which the factors are more explicitly described.  
\end{remark} 


Mixing and iterating Proposition~\ref{connsumcoH} and Theorem ~\ref{gyrationcoH} leads to more 
examples of Poincar\'{e} Duality complexes in $\mathcal{A}$.

\begin{example} Let $M$ and $N$ be simply-connected $n$-dimensional 
Poincar\'{e} Duality complexes. Suppose that $M\in\mathcal{A}$ and $\overline{N}$ 
is a co-$H$-space. Let $\tau,\omega:S^{k-1}\rightarrow SO(n)$. Then $\mathcal{G}^{k}_{\tau}(M)\in\mathcal{A}$ by Theorem~\ref{gyrationcoH} 
and $\mathcal{G}^{k}_{\omega}(N)$ has the property that $\overline{\mathcal{G}^{k}_{\omega}(N)}$ 
is a co-$H$-space by Lemma~\ref{pregyrationcoH}. Both gyrations are $(n+k-1)$-dimensional, 
so their connected sum exists, and Proposition~\ref{connsumcoH} implies that 
$\mathcal{G}^{k}_{\tau}(M)\conn\mathcal{G}^{k}_{\omega}(N)\in\mathcal{A}$.  
\end{example} 

\begin{example} 
Let $M$ be a simply-connected Poincar\'{e} Duality complex in $\mathcal{A}$. Then 
for any $k,\ell\geq 2$ and $\tau,\omega:S^{k-1} \rightarrow SO(n)$, by Theorem~\ref{gyrationcoH} the iterated gyration 
$\mathcal{G}^{\ell}_{\omega}(\mathcal{G}^{k}_{\tau}(M))$ is in $\mathcal{A}$. 
\end{example} 

A systematic family of mixed and iterated examples is the following. 
\medskip 

\noindent 
\textbf{Simply-connected $6$-manifolds with a regular circle action}. 
Duan~\cite{D} completed the classification of simply-connected $6$-manifolds 
with a regular circle action begun by Goldstein and Lininger~\cite{GL}. This begins 
with the classification of simply-connected $5$-manifolds by Smale~\cite{Sm} 
and Barden~\cite{Bar}. They showed that all simply-connected $5$-manifolds 
can be described up to diffeomorphism as iterated connected sums of five basic 
types: $S^{2}\times S^{3}$, $S^{2}\widetilde{\times} S^{3}$, $W$, $M_{k}$ and $X_{2^{i}}$. 
Here, $S^{2}\widetilde{\times} S^{3}$ is the nontrivial $S^{3}$-bundle over $S^{2}$, 
$W$ is the Wu manifold $SU(3)/SO(3)$, $M_{k}$ is a $5$-cell complex with 
$H_{2}(M)\cong\mathbb{Z}/k\mathbb{Z}\oplus\mathbb{Z}/k\mathbb{Z}$, and 
$X_{2^{i}}$ is also a $5$-cell complex with 
$H_{2}(X_{2^{i}})\cong\mathbb{Z}/2^{i}\mathbb{Z}\oplus\mathbb{Z}/2^{i}\mathbb{Z}$ 
but has a different attaching map for the $5$-cell than $M_{2^{i}}$. For any such 
$5$-manifold $M$, with $k=2$ there are two inequivalent twisted gyrations, the trivial 
one $\mathcal{G}^{2}_{0}(M)$ and a nontrivial one denoted $\mathcal{G}^{2}_{1}(M)$, 
corresponding to the trivial and nontrivial homotopy classes for maps 
\(\namedright{S^{1}}{}{SO(5)}\).  

\begin{theorem}[\cite{D} Theorem C] 
   \label{6manifold} 
   If $M$ is a simply-connected $6$-manifold that admits a regular circle action then 
   \[M\cong\left\{\begin{array}{l} 
         (S^{3}\times S^{3})\conn_{r}\mathcal{G}^{2}_{0}(S^{2}\times S^{3})\conn_{1\leq j\leq t} 
              \mathcal{G}^{2}_{1}(M_{k_{j}})\conn\mathcal{G}^{2}_{1}(H)\ \mbox{if $w_{2}(M)\neq 0$} \\ 
         (S^{3}\times S^{3})\conn_{r}\mathcal{G}^{2}_{0}(S^{2}\times S^{3})\conn_{1\leq j\leq t} 
              \mathcal{G}^{2}_{1}(M_{k_{j}})\ \mbox{if $w_{2}(M)= 0$} 
       \end{array}\right.\] 
    where $H\in\{S^{2}\widetilde{\times} S^{3},W,X_{k}\}$, $\conn_{r}(\cdots)$ means take the 
    connected sum $r$ times and $\conn_{0}(\cdots)$ means take the connected sum 
    with $S^{6}$.~$\qqed$ 
\end{theorem} 

We show that each manifold in Theorem~\ref{6manifold} has $\overline{M}\in\mathcal{M}$, 
and if there is at least one $\mathbb{Z}$-summand in $H_{2}(M)$, then $M\in\mathcal{A}$, 
implying that the refined homotopy decomposition in 
Corollary~\ref{cor:MbarwedgeofsphereandMoore} holds.  

\begin{proposition} 
   \label{6manifoldA} 
   Let $M$ be a simply-connected $6$-manifold that admits a regular circle action. 
   Then the following hold: 
   \begin{itemize} 
      \item[(a)] $\overline{M}\in\mathcal{M}$; 
      \item[(b)] if the rank of $H_{2}(M)$ is at least $1$ then $M\in\mathcal{A}$. 
   \end{itemize} 
\end{proposition} 

\begin{proof}  
First consider $\overline{M}$. In general, if $N_{1}$ and $N_{2}$ are $n$-dimensional 
closed manifolds then there is a homotopy equivalence 
$\overline{N_{1}\conn N_{2}}\simeq\overline{N}_{1}\vee\overline{N}_{2}$. 
Thus the homotopy type of $\overline{M}$ for each $M$ in Theorem~\ref{6manifold} is 
given by the wedge sum of the $5$-skeletons of each connected sum factor. By 
Lemma~\ref{BasuGhosh}, $\overline{\mathcal{G}^{2}_{t}(N)}\simeq\overline{N}\rtimes S^{1}$. 
If $\overline{N}$ is a co-$H$-space then 
$\overline{N}\rtimes S^{1}\simeq\overline{N}\vee(\overline{N}\wedge S^{1})= 
    \overline{N}\vee\Sigma\overline{N}$. 
Thus we obtain: 
\begin{itemize} 
   \item[] $\overline{(S^{3}\times S^{3})}\simeq S^{3}\vee S^{3}$; 
   \item[] $\overline{\mathcal{G}^{2}_{0}(S^{2}\times S^{3})}\simeq 
                 (S^{2}\vee S^{3})\vee\Sigma(S^{2}\vee S^{3})$;  
  \item[] $\overline{\mathcal{G}^{2}_{1}(M_{k_{j}})}\simeq 
                (P^{3}(k_{j})\vee P^{3}(k_{j}))\vee\Sigma(P^{3}(k_{j})\vee P^{3}(k_{j}))$; 
  \item[] $\overline{\mathcal{G}^{2}_{1}(S^{2}\widetilde{\times} S^{3})}\simeq 
                (S^{2}\vee S^{3})\vee\Sigma(S^{2}\vee S^{3})$; 
  \item[] $\overline{\mathcal{G}^{2}_{1}(W)}\simeq P^{3}(2)\vee\Sigma P^{3}(2)$; 
  \item[] $\overline{\mathcal{G}^{2}_{1}(X_{k})}\simeq 
                (P^{3}(k)\vee P^{3}(k))\vee\Sigma(P^{3}(k)\vee P^{3}(k))$. 
\end{itemize} 
In particular, this implies that each $\overline{M}$ is homotopy equivalent to a wedge 
of spheres and Moore spaces, so $\overline{M}\in\mathcal{M}$, proving part~(a). 

The hypothesis that the rank of $H_{2}(M)$ is at least~$1$ states  
that $H_{2}(M)$ has at least one $\mathbb{Z}$-summand. For degree reasons, 
this is equivalent to stating that $H_{2}(\overline{M})$ has at least one $\mathbb{Z}$-summand. 
The only wedge summands above that could satisfy this are  
$\overline{\mathcal{G}^{2}_{0}(S^{2}\times S^{3})}$ and 
$\overline{\mathcal{G}^{2}_{1}(S^{2}\widetilde{\times} S^{3})}$. Thus at least one 
of $\mathcal{G}^{2}_{0}(S^{2}\times S^{3})$ or $\mathcal{G}^{2}_{1}(S^{2}\widetilde{\times} S^{3})$ 
is a connected sum factor of $M$. In general, if $N_{1}$ and $N_{2}$ are $n$-dimensional 
closed manifolds then there is a map 
\(\namedright{N_{1}\conn N_{2}}{}{N_{1}}\) 
given by collapsing $\overline{N}_{2}$ to a point. In our case, there must be a map
\[\namedright{M}{}{\mathcal{G}^{2}_{0}(S^{2}\times S^{3})}\qquad\mbox{or}\qquad 
     \namedright{M}{}{\mathcal{G}^{2}_{1}(S^{2}\widetilde{\times} S^{3})}.\] 
Since $S^{2}\times S^{3}\in\mathcal{A}$, Theorem~\ref{gyrationcoH} implies that 
$\mathcal{G}^{2}_{0}(S^{2}\times S^{3})\in\mathcal{A}$. Therefore, there is a map 
\(\namedright{S^{2}}{}{\mathcal{G}^{2}_{0}(S^{2}\times S^{3})}\)
with a left homotopy inverse. Since this map factors through 
$\overline{\mathcal{G}^{2}_{0}(S^{2}\times S^{3})}$, we obtain a composite 
\[S^{2}\hookrightarrow\namedddright{\overline{\mathcal{G}^{2}_{0}(S^{2}\times S^{3})}}{} 
     {\overline{M}}{}{M}{}{\mathcal{G}^{2}_{0}(S^{2}\times S^{3})}\] 
that has a left homotopy inverse. Hence the map 
\(\namedright{S^{2}}{}{M}\) 
has a left homotopy inverse, implying that $M\in\mathcal{A}$. The argument 
in the case of $\mathcal{G}^{2}_{1}(S^{2}\widetilde{\times} S^{3})$ is similar since 
the bundle map 
\(\namedright{S^{2}\widetilde{\times} S^{3}}{}{S^{2}}\) 
is a left homotopy inverse for the inclusion of the bottom cell, implying that 
$S^{2}\widetilde{\times} S^{3}\in\mathcal{A}$. 
\end{proof} 

In particular, Proposition~\ref{6manifoldA}~(b) implies that the attaching map for the 
top cell of such an $M$ is integrally inert. This extends the known families of manifolds 
that have this property.

\section{Local decompositions of highly connected \texorpdfstring{$CW$}{CW}-complexes}
\label{sec:localdecomphighlyconnCW}

To further expand the examples to which we can apply Corollary ~\ref{cor:MbarwedgeofsphereandMoore}, we localise. Before considering Poincare duality complexes we first consider wedge decompositions of certain highly connected CW-complexes, when localised away from an explicit, finite set of primes. 

\subsection{Homotopy classes of maps involving spheres and Moore spaces}
We start with classical results of Serre \cite{Ser} about torsion in the homotopy groups of spheres. Let $p$ be a prime.

\begin{theorem}
    \label{thm:1stptorssphere}
    Let $m \geq 2$. The group $\pi_k(S^m)$ is torsion except when $k=m$ or when $m$ is even and $k=2m-1$. Further, $\pi_{k}(S^{m})$ contains no $p$-torsion for $k < m+2p-3$. $\qqed$
\end{theorem}

Next, a corresponding result is given for the homotopy groups of Moore spaces. This is an extension of an argument by Cutler and So \cite[Lemma 2.2]{CS}. For a prime $p$ and integers $r\geq 1$ and $m\geq 2$, the \emph{mod-$p^{r}$ Moore space} $P^{m}(p^{r})$ is the homotopy cofibre of the degree $p^{r}$ map on $S^{m-1}$. 

\begin{lemma}
\label{lem:htpygrpsMoore}
Fix $m \geq 3$. Let $k$ satisfy $m+1 \leq k \leq 2m-2$ and let $p$ be a prime such that $p > \frac{k-m+3}{2}$. Then $\pi_k(P^{m+1}(p^r))$ is trivial. 
\end{lemma}
\begin{proof}
    Localise at $p$. Let $q:P^{m+1}(p^r) \rightarrow S^{m+1}$ be the pinch map to the top cell and let $F$ be its homotopy fibre. By \cite[p.138]{N}, $H_i(F) \cong \mathbb{Z}$ if $i = lm$ for $l \geq 0$ and is trivial otherwise. Hence, the $(2m-1)$-skeleton of $F$ is homotopy equivalent to $S^m$. 

    Let $f:S^k \rightarrow P^{m+1}(p^r)$ be a map and consider the composite 
    \(\nameddright{S^{k}}{f}{P^{m+1}(p^{r})}{q}{S^{m+1}}\). 
    Since $P^{m+1}(p^{r})$ is rationally contractible, $q\circ f$ represents a torsion homotopy class in $\pi_{k}(S^{m+1})$. The hypothesis $p>\frac{k-m+3}{2}$ implies that $k<m+2p-3$, 
    and therefore $k<(m+1)+2p-3$, so Theorem ~\ref{thm:1stptorssphere} implies that $\pi_k(S^{m+1})$ contains no $p$-torsion. Thus $q \circ f$ is null homotopic, implying that there is a lift \[\begin{tikzcd}
	{S^k} & F \\
	& {P^{m+1}(p^r)}
	\arrow["\phi", from=1-1, to=1-2]
	\arrow["f", from=1-1, to=2-2]
	\arrow[from=1-2, to=2-2]
\end{tikzcd}\] 
for some map $\phi$. By hypothesis, $k\leq 2m-2$, so $\phi$ factors through the $(2m-2)$-skeleton of $F$, which is $S^{m}$. Thus $\phi$ factors through a map 
\(\phi'\colon\namedright{S^{k}}{}{S^{m}}\). 
Since $k<2m-1$, $\phi'$ represents a torsion class in $\pi_{k}(S^{m})$. Since $m \geq 3$ and  the hypothesis $p>\frac{k-m+3}{2}$ implies that $k<m+2p-3$,  Theorem ~\ref{thm:1stptorssphere} implies that $\phi$ is null homotopic. Hence, $f$ is null homotopic.
\end{proof}

We now turn to homotopy classes of maps where the domain is a Moore space. 

\begin{definition}
    Let $p \geq 3$ be a prime, $r \geq 1$ and $X$ be a space. The \emph{$n^{th}$ homotopy group of $X$ with coefficients} in $\mathbb{Z}/p^r\mathbb{Z}$ is the group \[\pi_n(X;\mathbb{Z}/p^r\mathbb{Z}) := [P^n(p^r),X].\]
\end{definition}

Analogues of Theorem ~\ref{thm:1stptorssphere} and Lemma ~\ref{lem:htpygrpsMoore} will now be proved for homotopy groups with coefficients. This requires a universal coefficient theorem for homotopy groups that can be found, for example, in \cite[1.3.1]{N}.

\begin{theorem}
    \label{thm:univcoeffhtpy}
    Let $k \geq 2$ and $X$ be a space. There is a short exact sequence \[0 \rightarrow \pi_k(X) \otimes \mathbb{Z}/p^r\mathbb{Z} \rightarrow \pi_k(X;\mathbb{Z}/p^r\mathbb{Z}) \rightarrow \mathrm{Tor}(\pi_{k-1}(X),\mathbb{Z}/p^r\mathbb{Z}) \rightarrow 0.\eqno\qqed\]
\end{theorem}

\begin{lemma}
    \label{lem:htpygrpcoeffsphere}
    Let $m \geq 3$, $r \geq 1$, $m+1 \leq k \leq 2m-2$, and $p$ be a prime such that $p > \frac{k-m+3}{2}$. Then $\pi_k(S^m;\mathbb{Z}/p^r\mathbb{Z})$ is trivial.
\end{lemma}
\begin{proof}
    By Theorem ~\ref{thm:univcoeffhtpy}, there is a short exact sequence \[0 \rightarrow \pi_k(S^m) \otimes \mathbb{Z}/p^r\mathbb{Z} \rightarrow \pi_k(S^m;\mathbb{Z}/p^r\mathbb{Z}) \rightarrow \mathrm{Tor}(\pi_{k-1}(S^m),\mathbb{Z}/p^r\mathbb{Z}) \rightarrow 0.\] First, consider the case where $k \geq m+2$. In this case, as $k\leq 2m-2$, both $\pi_k(S^m)$ and $\pi_{k-1}(S^m)$ are torsion groups and, as $p>\frac{k-m+3}{2}$, Lemma ~\ref{thm:1stptorssphere} implies both groups contain no $p$-torsion. Hence, $\pi_k(S^m) \otimes \mathbb{Z}/p^r\mathbb{Z}$ and $\mathrm{Tor}(\pi_{k-1}(S^m),\mathbb{Z}/p^r\mathbb{Z})$ are trivial, implying that $\pi_k(S^m;\mathbb{Z}/p^r\mathbb{Z})$ is trivial.

    If $k = m+1$ then $\pi_{m+1}(S^m)$ is $2$-torsion since $m \geq 3$. The hypothesis that $p>\frac{k-m+3}{2}$ is equivalent to $p>2$ for $k=m+1$, so $\pi_{m+1}(S^m) \otimes \mathbb{Z}/p^r\mathbb{Z}$ is trivial. The Tor term is trivial since $\pi_m(S^m) \cong \mathbb{Z}$. Hence, $\pi_{m+1}(S^m;\mathbb{Z}/p^r\mathbb{Z})$ is trivial.
\end{proof}

\begin{lemma}
      \label{lem:htpygrpcoeffMoore}
         Let $m \geq 3$, $r,s \geq 1$, $m+2 \leq k \leq 2m-2$, and $p,q$ be primes. If $p \neq q$, then $\pi_k(P^{m+1}(p^r);\mathbb{Z}/q^s\mathbb{Z})$ is trivial. If $p = q$, suppose that $p> \frac{k-m+3}{2}$. Then $\pi_k(P^{m+1}(p^r);\mathbb{Z}/p^s\mathbb{Z})$ is trivial.
\end{lemma}
\begin{proof}
    By Theorem ~\ref{thm:univcoeffhtpy}, there is a short exact sequence \[0 \rightarrow \pi_k(P^{m+1}(p^r)) \otimes \mathbb{Z}/q^s\mathbb{Z} \rightarrow \pi_k(P^{m+1}(p^r);\mathbb{Z}/q^s\mathbb{Z}) \rightarrow \mathrm{Tor}(\pi_{k-1}(P^{m+1}(p^r)),\mathbb{Z}/q^s\mathbb{Z}) \rightarrow 0.\] Since $P^{m+1}(p^{r})$ is contractible when localised at any prime not equal to $q$, the homotopy groups of $P^{m+1}(p^{r})$ are all $p$-torsion. Therefore if $p \neq q$ then both $\pi_{k}(P^{m+1}(p^{r}))\otimes\mathbb{Z}/q^{s}\mathbb{Z}$ and the Tor term are trivial, implying that $\pi_k(P^{m+1}(p^r);\mathbb{Z}/q^s\mathbb{Z})$ is trivial.

    If $p = q$ then the hypotheses on both $k$ and $p$ imply, by Lemma ~\ref{lem:htpygrpsMoore}, that both $\pi_k(P^{m+1}(p^r))$ and $\pi_{k-1}(P^{m+1}(p^r))$ are trivial. Hence, $\pi_k(P^{m+1}(p^r);\mathbb{Z}/p^s\mathbb{Z})$ is trivial.
\end{proof}

\subsection{Local decompositions of highly connected \texorpdfstring{$CW$}{CW}-complexes}

We now give decompositions of certain highly connected $CW$-complexes, after localisation away from sufficiently many primes. Recall that $\mathcal{W}$ is the collection of topological spaces homotopy equivalent to a finite type wedge of spheres.

\begin{lemma}
    \label{lem:wedgeofsphereshighlyconnCW}
    Let $X$ be an $(m-1)$-connected $CW$-complex of dimension $n \leq 2m-1$, where $m \geq 2$. Localise away from primes appearing as $p$-torsion in $H_*(X)$ and primes $p \leq \frac{n-m+3}{2}$. Then $X \in \mathcal{W}$.
\end{lemma}
\begin{proof}
By assumption on the dimension of $X$, the attaching map for each cell is in the stable range. Therefore, integrally, $X \simeq \Sigma X'$ for some $CW$-complex $X'$. Rationally, any suspension is homotopy equivalent to a wedge of spheres. Therefore, localised away from primes appearing as $p$-torsion in $H_*(X)$ and primes $p \leq \frac{n-m+3}{2}$, \cite[Lemma 5.1]{HT3} implies that $X \in \mathcal{W}$.
\end{proof}

A torsion analogue of Lemma ~\ref{lem:wedgeofsphereshighlyconnCW} can be proved if we restrict connectivity and dimension. This requires an extension of the argument in \cite[Lemma 5.1]{HT3}. Recall that $\mathcal{M}$ is the collection of topological spaces homotopy equivalent to a finite type wedge of spheres and Moore spaces.

\begin{lemma}
    \label{lem:highlyconnCWtorsWandM}
    Let $X$ be an $(m-1)$-connected $CW$-complex of dimension $n$, where $m \geq 3$ and $n \leq 2m-2$. Localise away from primes $p \leq \frac{n-m+3}{2}$. Then $X \in \mathcal{M}$.
\end{lemma}
\begin{proof}
        Since $X$ is simply connected, it has a homology decomposition (see \cite[Chapter 4.H]{Ha}, for example), which is a sequence of homotopy cofibrations \[M_t \xrightarrow{f_t} X_{t-1} \rightarrow X_t,\] for $2 \leq t \leq n$, with $X_n = X$, each $M_t$ is a wedge of $(t-1)$-dimensional spheres and $t$-dimensional Moore spaces, and $f_t$ is homologically trivial. Since $X$ is $(m-1)$-connected, $X_1,\cdots,X_{m-1} = *$.
        
        The proof is by induction. Localise away from primes $p \leq \frac{n-m+3}{2}$. When $t = m$, we obtain a homotopy cofibration \[\bigvee\limits_{i=1}^{l_m} S^{m-1} \vee \bigvee\limits_{j=1}^{l'_m}P^{m}(p_j^{r_j}) \rightarrow * \rightarrow X_m,\] implying that $X_m \in \mathcal{M}$. Suppose that for $t < s$, \[X_t \simeq \bigvee\limits_{i=1}^{k_t} S^{n_i} \vee \bigvee\limits_{j=1}^{k'_t} P^{n'_j+1}(q_j^{r'_j}),\] where $m \leq n_i, n'_j \leq t$ for each $i,j$. For $t=s$, there is a homotopy cofibration \[\bigvee\limits_{\imath=1}^{l_s} S^{s-1} \vee \bigvee\limits_{\jmath=1}^{l'_s} P^{s}(p_\jmath^{r_\jmath}) \xrightarrow{f_s} \bigvee\limits_{i=1}^{k_{s-1}} S^{n_i} \vee \bigvee\limits_{j=1}^{k'_{s-1}} P^{n'_j+1}(q_j^{r'_j})  \rightarrow X_s.\] Since $s \leq 2m-2$ and each $n_{i},n'_{j}\leq s-1$, the Hilton-Milnor Theorem implies that \[f_s \simeq \sum_{i=1}^{k_{s-1}} f_s^i + \sum_{j=1}^{k'_{s-1}} g_s^j,\] where $f_{s}^i$ is the composite \[f^i_s: \bigvee\limits_{\imath=1}^{l_s} S^{s-1} \vee \bigvee\limits_{\jmath=1}^{l'_s} P^{s}(p_\jmath^{r_\jmath})\xrightarrow{f_s} \bigvee\limits_{i=1}^{k_{s-1}} S^{n_i} \vee \bigvee\limits_{j=1}^{k'_{s-1}} P^{n'_j+1}(q_i^{r'_j}) \xrightarrow{p_i} S^{n_i}\] and $p_i$ is the pinch map, while $g_{s}^j$ is the composite \[g_s^j:\bigvee\limits_{\imath=1}^{l_s} S^{s-1} \vee \bigvee\limits_{\jmath=1}^{l'_s} P^{s}(p_\jmath^{r_\jmath}) \xrightarrow{f_s} \bigvee\limits_{i=1}^{k_{s-1}} S^{n_i} \vee \bigvee\limits_{j=1}^{k'_{s-1}} P^{n'_j+1}(q_j^{r'_j}) \xrightarrow{p'_j} P^{n'_j+1}(q_j^{r'_j})\] and $p'_j$ is the pinch map. 
        
        Consider $f_s^i$. If $n_i = s-1$ then, as $f_s$ is homologically trivial, $f_s^i$ is null homotopic by the Hurewicz isomorphism. If $n_i < s-1$ then, localised away from primes $p \leq \frac{n-m+3}{2}$, Theorem ~\ref{thm:1stptorssphere} implies that $\pi_{s-1}(S^{n_i})$ is trivial and Lemma ~\ref{lem:htpygrpcoeffsphere} implies that $\pi_s(S^{n_i};\mathbb{Z}/p_{\jmath}^{r_{\jmath}}\mathbb{Z})$ is trivial. Hence $f_s^i$ is null homotopic. 

      Now consider $g_s^j$. If $n'_j = s-1$ then, as $f_s$ is homologically trivial, $g_s^j$ is null homotopic by the Hurewicz isomorphism. If $n'_j < s-1$ then, localised away from primes $p \leq \frac{n-m+3}{2}$, Theorem ~\ref{lem:htpygrpsMoore} implies that $\pi_{s-1}(P^{n_j'+1}(p_j^{r_j}))$ is trivial and Lemma ~\ref{lem:htpygrpcoeffsphere} implies that $\pi_s(P^{n_j'+1}(p_j^{r_j});\mathbb{Z}/p_{\jmath}^{r_{\jmath}}\mathbb{Z})$ is trivial. Hence $g_s^j$ is null homotopic. 
      
      Therefore $f_s$ is null homotopic, implying that $X \in \mathcal{M}$.
\end{proof}

\section{Highly connected Poincar\'e duality complexes and moment-angle manifolds associated to minimally non-Golod complexes}
\label{sec:highyconnPDcomplex}

In this section, we extend the reach of Corollary~\ref{cor:MbarwedgeofsphereandMoore} to Poincar\'{e} duality complexes that satisfy its hypotheses but only after localisation away from a finite set of primes. 

\noindent
\textbf{Highly connected Poincar\'e duality complexes}. Let $M$ be an $(m-1)$-connected Poincar\'e duality complex such that $m\geq 2$ and $H_*(\overline{M})$ contains a $\mathbb{Z}$ summand. Let $k$ be the least dimension of such a $\mathbb{Z}$ summand. For Poincar\'e duality complexes in a certain dimensional range, we will show that after sufficient localisation: 
\begin{enumerate} 
    \item $\overline{M} \in \mathcal{M}$; 
    \item there is a map $M \rightarrow S^k$ that has a right homotopy inverse. 
\end{enumerate} The results in Section ~\ref{sec:localdecomphighlyconnCW} will be used to show that $(1)$ holds. To show that $(2)$ holds, we use the following slight reformulation of  two statements in \cite[Theorems 6.3 and 7.5]{T2}, the proofs of which are easily adapted to the case below.

\begin{theorem}
\label{thm:localLHIsphere}
    Let $M$ be an $(m-1)$-connected Poincar\'e duality complex of dimension $n$, where $2 \leq m < n$. Let $k$ be the least integer such that $H_k(M)$ contains a $\mathbb{Z}$ summand and suppose that $k < n$. If $k$ is even, suppose there exists a generator $x \in H^k(M)$ such that $x^2 = 0$. Localise away from primes $p$ appearing as $p$-torsion in $H_i(M)$ with $i < k$ and primes $p \leq \frac{n-k+3}{2}$ if $k$ is odd, or primes $p \leq \frac{n-k+4}{2}$ if $k$ is even. Then there exists a map $S^k \rightarrow M$ with a left homotopy inverse and the loop map 
    \(\namedright{\Omega\overline{M}}{}{\Omega M}\) 
    has a right homotopy inverse.
\end{theorem}
\begin{proof}
    The case where $k$ is odd follows from \cite[Lemma 6.1]{T2} and \cite[Theorem 6.3]{T2}. The case where $k$ is even follows from \cite[Lemma 7.4]{T2} and \cite[Theorem 7.5]{T2}.
\end{proof}

\begin{example}
    Connecting back to Section~\ref{sec:connsumgyration}, let $M$ be an $(m-1)$-connected Poincar\'e duality complex of dimension $n$, where $2 \leq m < n$. Localise away from the primes in the statement of Theorem ~\ref{thm:localLHIsphere}. Then there is a map $S^k \rightarrow M$ with a left homotopy inverse. For any $k \geq 2$ and any~$\tau$, by Theorem~\ref{gyrationcoH} there is a map $S^k \rightarrow \mathcal{G}_{\tau}^k(M)$ with a left homotopy inverse. The attaching map of the top cell of $\mathcal{G}_{\tau}^k(M)$ is then inert by \cite[Theorem 1.1]{T2}.
\end{example}

\begin{theorem}
    \label{thm:PDcomplxcup2}
    Let $M$ be an $(m-1)$-connected Poincar\'e duality complex of dimension $n$, where $2 \leq m < n$ and $n \leq 3m-1$. Let $k$ be the least integer such that $H_k(M)$ contains a $\mathbb{Z}$ summand, and suppose that $k < n$. If $k$ is even and $k=m = n-m$, suppose there exists a generator $x \in H^k(M)$ such that $x^2 = 0$. Localise away from primes $p$ appearing as $p$-torsion in $H_*(M)$ and primes $p \leq \frac{n-k+3}{2}$ if $k$ is odd, or primes $p \leq \frac{n-k+4}{2}$ if $k$ is even. Then there is a homotopy fibration 
   \[\nameddright{A\vee(B\wedge\Omega S^{m})}{}{M}{h'}{S^{k}},\] where $A$ and $B$ are wedges of spheres that can be explicitly enumerated as in Corollary ~\ref{cor:MbarwedgeofsphereandMoore}.
   Moreover, this homotopy fibration splits after looping to give a homotopy equivalence 
   \[\Omega M\simeq\Omega S^{k}\times\Omega(A\vee(B\wedge\Omega S^{k})).\]
\end{theorem}

\begin{proof}
The dimension restriction on~$M$ implies that, by Poincar\'e duality and the universal coefficient theorem, $\overline{M}$ is a $(m-1)$-connected $CW$-complex of dimension $d \leq 2m-1$. Localise away from the primes in the statement of the theorem. By Lemma ~\ref{lem:wedgeofsphereshighlyconnCW}, $\overline{M}\in \mathcal{W}$, and Theorem ~\ref{thm:localLHIsphere} implies that there exists a map $S^k \rightarrow M$ with a left homotopy inverse. Hence, Corollary ~\ref{cor:MbarwedgeofsphereandMoore} and Remark ~\ref{localMfib} imply the existence of the asserted homotopy fibration and loop space decomposition.
\end{proof}

In \cite{BB2}, an explicit loop space decomposition of $(m-1)$-connected Poincar\'e duality complexes of dimension $n \leq 3m-2$ was given after localisation away from a finite set of primes. This was obtained by giving a presentation of $H_*(\Omega M)$ as a quadratic algebra and using an explicit basis of this algebra to define a map from a product of looped spheres to $\Omega M$ which was shown to be a homotopy equivalence. Our approach reverses this: we first find a homotopy equivalence for $\Omega M$ in the slightly greater range $n\leq 3m-1$ and use this to calculate $H_{\ast}(\Omega M)$. 

\begin{theorem}
    \label{thm:PDcmplxcuplength2loophom}
     Let $M$ be an $(m-1)$-connected, closed Poincar\'e duality complex of dimension $n$, where $2 \leq m < n$ and $n \leq 3m-1$. Let $k$ be the least integer such that $H_k(M)$ contains a $\mathbb{Z}$ summand, and suppose that $k < n$. If $k$ is even and $k=m = n-m$, suppose there exists a generator $x \in H^k(M)$ such that $x^2 = 0$. Localise away from primes $p$ appearing as $p$-torsion in $H_*(M)$ and primes $p \leq \frac{n-k+3}{2}$ if $k$ is odd, or primes $p \leq \frac{n-k+4}{2}$ if $k$ is even. Then there is an isomorphism of Hopf algebras \[H_*(\Omega M) \cong T(u_1,\cdots,u_l)/(I),\] where $u_1,\cdots,u_l$ correspond to $\mathbb{Z}$ summands in $H_*(\overline{M};\mathbb{Z})$ and $I$ is a sum of monomials in $u_1,\cdots,u_l$.
\end{theorem}

\begin{proof}
    Localise away from $p$-torsion in $H_*(M)$ and primes $p \leq \frac{n-k+3}{2}$ if $k$ is odd, or primes $p \leq \frac{n-k+4}{2}$ if $k$ is even. By Lemma ~\ref{lem:wedgeofsphereshighlyconnCW}, $\overline{M}\in \mathcal{W}$. Write $\overline{M} \simeq \bigvee\limits_{i=1}^{l}S^{n_i}$, where $k \leq n_i \leq n-k \leq 2m-1$. There is a homotopy cofibration \[S^{n-1} \xrightarrow{f} \bigvee\limits_{i=1}^{l}{S^{n_i}} \xrightarrow{i} M,\] where $f$ attaches the $n$-cell to $M$ and $i$ is the inclusion of the $(n-1)$-skeleton. 

   By Theorem ~\ref{thm:localLHIsphere}, $\Omega i$ has a right homotopy inverse. Therefore Theorem ~\ref{Minertloophlgy} implies there is an isomorphism of Hopf algebras \[H_*(\Omega M) \cong T(u_1,\cdots,u_l)/(\mathrm{Im}(\tilde{f}_*)),\] where $|u_i| = n_i-1$. It remains to show that $\mathrm{Im}(\tilde{f}_*)$ is generated by a sum of monomials in $u_1,\cdots,u_l$. 
   
   Let $p_i:\bigvee\limits_{i=1}^{l}{S^{n_i}} \rightarrow S^{n_i}$ be the pinch map and let $k_i:S^{n_i} \rightarrow \bigvee\limits_{i=1}^{l}{S^{n_i}}$ be the inclusion. Since $n \leq 3m-1$, the Hilton-Milnor Theorem implies that \[f \simeq \sum _{i=1}^{l} g_i + \sum_{j=1}^{l'} W^2_j \circ h_j + \sum_{j=1}^{l'} W^3_j \circ h'_j\] where: \begin{enumerate}
       \item $g_i$ is the composite $p_i \circ f$;
       \item each $W^2_j:S^{n_1+n_2-1} \rightarrow \bigvee_{i=1}^l S^{n_i}$ is a Whitehead product of the form $[k_{i_1},k_{i_2}]$;
       \item $h_j \in \pi_{n-1}(S^{n_{i_1}+n_{i_2}-1})$;
       \item each $W^3_j:S^{n_1+n_2+n_3-2} \rightarrow \bigvee_{i=1}^l S^{n_i}$ is a Whitehead product of length $3$ involving the maps $k_{i_1}$, $k_{i_2}$ and $k_{i_3}$;
       \item $h'_j \in \pi_{n-1}(S^{n_{i_1}+n_{i_2}+n_{i_3}-2})$.
   \end{enumerate} 

   Consider $h_j$. By assumption $n-1 < 3m-2$ and $n_{i_1}+n_{i_2}-1 \geq 2m+1$, implying that $n-1 < 2(n_{i_1}+n_{i_2}-1)-1$. Therefore, $\pi_{n-1}(S^{n_{i_1}+n_{i_2}-1})$ is torsion unless $n-1 = n_{i_1}+n_{i_2}-1$. Since we have localised away from primes $p \leq \frac{n-k+3}{2}$ if $k$ is odd, or primes $p \leq \frac{n-k+4}{2}$ if $k$ is even, it follows from Theorem \ref{thm:1stptorssphere} that $h_j$ is null homotopic, unless $n-1 = n_{i_1}+n_{i_2}-1$ and $h_j$ is a multiple of the identity map. A similar argument shows that $h'_j$ is null homotopic, unless $n-1 = n_{i_1}+n_{i_2}+n_{i_3}-2$ and $h'_j$ is a multiple of the identity map. Therefore, $f \simeq \sum_{i=1}^{l} g_i + \sum_{j=1}^{l'} a_j W_j$, where each $W_j$ is a Whitehead product of length $2$ or $3$ involving the inclusions $k_{i_{j}}$ and $a_j \in \mathbb{Z}$.
   
   Under adjunction, $\widetilde{f} \simeq \sum_{i=1}^{l} \widetilde{g}_i + \sum_{j=1}^{l'} a_j \widetilde{W}_j$, where $\widetilde{g}_{i}$ and $\widetilde{W}_{j}$ are the adjoints of $g_{i}$ and $W_{j}$ respectively. Let $\lambda \in H_*(S^{n-1})$ be a generator. Since $n_{i}<n-1$, the map $g_{i}$ represents either a torsion homotopy class in $\pi_{n-1}(S^{n_{i}})$ or an integral summand if $n=2n_{i}$. But we have localised away from all primes that could contribute a torsion class by hypothesis, so either $g_{i}$ is null homotopic or it is a map of non-trivial Hopf invariant. Therefore, either $(\widetilde{g}_i)_{\ast}$ maps $\lambda$ to zero or to some multiple of $u_i^2$. Next, the adjoint of each inclusion $k_{i}$ has Hurewicz image $u_{i}$. Therefore, as $\widetilde{W}_j$ is a Samelson product of the adjoints of the inclusions $k_{i_{j}}$, in homology it sends a generator $\lambda \in H_*(S^{n-1})$ to a commutator in the $u_{i_{j}}$'s, implying that its image in homology is a sum of monomials in $u_1,\cdots,u_l$. Hence, $\mathrm{Im}(\widetilde{f}_*)$ is generated by a sum of monomials in $u_1,\cdots,u_l$.
\end{proof}

\begin{remark}
\label{rmk:quadrel}
Note that the proof of Theorem~\ref{thm:PDcmplxcuplength2loophom} strengthens if $n \leq 3m-2$. In that case, for dimensional reasons the decomposition of $f$ does not contain any Whitehead products of length~$\geq 3$, implying that $I$ is a quadratic relation.
\end{remark} 

If the dimension in Theorem ~\ref{thm:PDcomplxcup2} is slightly restricted then there is an analogous loop space decomposition that allows for large torsion in homology.

\begin{theorem}
\label{thm:loopPDcomplextors}
        Let $M$ be an $(m-1)$-connected Poincar\'e duality complex of dimension $n$, where $3 \leq m < n-m$ and $n \leq 3m-2$. Let $k$ be the least integer such that $H_k(M)$ contains a $\mathbb{Z}$ summand, and suppose that $k < n$. If $k$ is even and $k=m = n-m$, suppose there exists a generator $x \in H^k(M)$ such that $x^2 = 0$. Localise away from primes $p$ appearing as $p$-torsion in $H_l(M)$ for $l < k$, and primes $p \leq \frac{n-k+3}{2}$ if $k$ is odd, or primes $p \leq \frac{n-k+4}{2}$ if $k$ is even. Then there is a homotopy fibration 
   \[\nameddright{A\vee(B\wedge\Omega S^{k})}{}{M}{h'}{S^{k}},\] where $A$ and $B$ are wedges of spheres and Moore spaces that can be explicitly enumerated as in Corollary ~\ref{cor:MbarwedgeofsphereandMoore}.
   Moreover, this homotopy fibration splits after looping to give a homotopy equivalence 
   \[\Omega M\simeq\Omega S^{k}\times\Omega(A\vee(B\wedge\Omega S^{k})).\]
\end{theorem}
\begin{proof} 
The dimension restriction on $M$ implies that, by Poincar\'e duality and the universal coefficient theorem, $\overline{M}$ is a $(m-1)$-connected $CW$-complex of dimension $d \leq 2m-2$. Localise away from the primes appearing in the statement of the theorem. By Lemma ~\ref{lem:highlyconnCWtorsWandM}, $\overline{M}\in \mathcal{M}$, and Theorem ~\ref{thm:localLHIsphere} implies that there is a map $S^k \rightarrow M$ with a left homotopy inverse. Hence, Corollary ~\ref{cor:MbarwedgeofsphereandMoore} and Remark ~\ref{localMfib} imply the result.
 \end{proof}

\noindent
\textbf{Moment-angle manifolds associated to minimally non-Golod complexes}. We extend the family of moment-angle manifolds for which we can apply Corollary ~\ref{cor:MbarwedgeofsphereandMoore} by localising. A simplicial complex $K$ on $[m]$ is called \emph{Golod} over a field $\mathbb{K}$ if all cup products and higher Massey products in $H^*(\zk;\mathbb{K})$ are trivial, and $K$ is \emph{minimally non-Golod} if $K\setminus i$ is Golod for all $i \in [m]$. Examples of minimally non-Golod complexes include the neighbourly spheres considered in Section ~\ref{sec:examples}, as proved in \cite[Proposition~3.6]{L} if $K$ is the dual of a boundary polytope and \cite[Theorem 6.2]{ST} in general.

Recall that if $K$ is a triangulation of $S^n$ on $m$ vertices then $\zk$ is a manifold of dimension $n+m+1$. There is a homotopy cofibration \[S^{n+m} \xrightarrow{f} \overline{\zk} \rightarrow \zk\] where $f$ is the attaching map of the top cell. We first show that if $K$ is minimally non-Golod then $\overline{\zk}$ is rationally homotopy equivalent to a wedge of spheres. 

\begin{lemma}
    \label{lem:ratminnonGolodskel}
    Let $K$ be a triangulation of $S^n$ on $[m]$ that is minimally non-Golod. Then, rationally, $\overline{\zk} \in \mathcal{W}$.  
\end{lemma}
\begin{proof}
     An unpublished result of Berglund (see \cite[Proposition 2.4]{St} for a proof) states that $\zk$ is rationally Golod if and only if it is rationally a co-$H$ space. Any co-$H$ space is rationally a wedge of spheres, so any $\zk$ associated to a Golod simplicial complex is rationally in $\mathcal{W}$.

    Since $K$ is minimally non-Golod, $K \setminus i$ is Golod for each $i \in [m]$, implying that $\mathcal{Z}_{K \setminus i} \in \mathcal{W}$. Since each $\mathcal{Z}_{K \setminus i} \in \mathcal{W}$, it follows from \cite[Theorem 6.1]{ST} that $\overline{\zk}$ is a co-$H$ space. Therefore $\overline{\zk}$ is rationally in $\mathcal{W}$.
\end{proof}

The rational result can be used to give a $p$-local decomposition of $\Omega \zk$ when $K$ is minimally non-Golod. Recall that if $K$ is $k$-neighbourly, then $\zk$ is $(2k+2)$-connected, and if there is a minimal missing face of dimension $k+1$, then there is a map $S^{2k+3} \rightarrow \zk$ with a left homotopy inverse. 

\begin{theorem}
    \label{thm:minnonGolod}
    Let $K \neq \partial{\Delta^{n+1}}$ be a triangulation of $S^n$ on $[m]$ that is $k$-neighbourly and minimally non-Golod. Suppose there is a minimal missing face of dimension $k+1$. Then localised away from primes $p$ appearing as $p$-torsion in $H_*(\zk)$ and primes $p \leq \frac{m+n-4k-2}{2}$, there is a homotopy fibration 
   \[\nameddright{A\vee(B\wedge\Omega S^{2k+3})}{}{\zk}{h'}{S^{2k+3}}\] 
   that splits after looping to give a homotopy equivalence 
   \[\Omega \zk\simeq\Omega S^{2k+3}\times\Omega(A\vee(B\wedge\Omega S^{2k+3})).\] The spaces $A$ and $B$ are wedges of spheres that can be explicitly enumerated as in Corollary ~\ref{cor:MbarwedgeofsphereandMoore}.
\end{theorem}
\begin{proof}
    By Lemma ~\ref{lem:ratminnonGolodskel}, rationally, $\overline{\zk} \in \mathcal{W}$. The space $\overline{\zk}$ is $(2k+2)$-connected and $(m+n-2k-2)$-dimensional. Lemma~\ref{lem:wedgeofsphereshighlyconnCW} implies that localised away from primes $p$ appearing as $p$-torsion in $H_*(\zk)$ and primes $p \leq \frac{m+n-4k-2}{2}$,  $\overline{\zk} \in \mathcal{W}$. Hence, the hypotheses of Corollary ~\ref{cor:MbarwedgeofsphereandMoore} are satisfied locally which gives the asserted result.
\end{proof}

\bibliographystyle{amsalpha}

\end{document}